\documentclass[12pt]{article}

\usepackage{geometry}
\usepackage{hyperref}

\usepackage{graphicx}
\usepackage{epstopdf}
\usepackage{bold-extra}
\usepackage{mathtools}
\usepackage{psfrag}
\usepackage{graphicx}
\usepackage{pgfplots}
\usepackage{amssymb}
\usepackage{subcaption}

\usepackage{xspace}

\usepackage{enumitem}
\setitemize{label=\scriptsize{$\blacksquare$}}
\setenumerate{label=(\alph*)}
\usepackage{amssymb,amsthm}

\newtheorem{theorem}{Theorem}
\newtheorem{lemma}[theorem]{Lemma}
\newtheorem{remark}[theorem]{Remark}

\newtheorem{proposition}[theorem]{Lemma}
\newtheorem{condition}[theorem]{Condition}

\numberwithin{equation}{section}
\numberwithin{figure}{section}
\numberwithin{theorem}{section}

\newcommand{\imi}{\mathsf{i}}
\newcommand{\mua}{\mu_a}
\newcommand{\mus}{\mu_s}
\newcommand{\barmua}{\bar{\mu}_a}
\newcommand{\barmus}{\bar{\mu}_s}
\newcommand{\qo}{f}
\newcommand{\qi}{q}
\newcommand{\fluence}{\phi}

\newcommand{\R}{\mathbb R}

\newcommand{\N}{\mathbb N}
\newcommand{\sph}{\mathbb S}
\newcommand{\edot}{\,\cdot\,}
\newcommand{\trans}{\mathsf{T}}
\newcommand{\lin}{\star}

\newcommand{\Wo}{\mathbf W}
\newcommand{\Ho}{\mathbf H}
\newcommand{\Vo}{\mathbf V}

\newcommand{\Fo}{\mathbf F}

\newcommand{\Ao}{\mathbf A}
\newcommand{\Ko}{\mathbf K}
\newcommand{\Io}{\mathbf I}
\newcommand{\Do}{\mathbf D}

\newcommand{\supp}{\mbox{supp}}
\newcommand{\rmd}{\mathrm d}
\newcommand{\eps}{\epsilon}
\newcommand{\mS}{\partial D}

\newcommand{\al}{\alpha}
\newcommand{\la}{\lambda}
\newcommand{\dom}{\mathcal M}
\newcommand{\diam}{\operatorname{diam}}
\newcommand{\dist}{\operatorname{dist}}

\newcommand\abs[1]{\left\vert#1\right\vert}
\newcommand\sabs[1]{\lvert#1\rvert}
\newcommand\norm[1]{\left\Vert#1\right\Vert}
\newcommand\snorm[1]{\Vert#1\Vert}

\newcommand\set[1]{\bigl\{#1\bigr\}}
\newcommand{\kl}[1]{\left(#1\right)}

\newcommand{\skl}[1]{(#1)}

\newcommand\sinner[2]{\langle#1,#2\rangle}

\makeatletter
\newcommand*\bigcdot{\mathpalette\bigcdot@{.6}}
\newcommand*\bigcdot@[2]{\mathbin{\vcenter{\hbox{\scalebox{#2}{$\m@th#1\bullet$}}}}}
\makeatother

\newcommand\ip[2]{{#1}\bigcdot {#2}}

\newcommand{\Om}{\Omega}
\newcommand{\La}{\Lambda}

\newcommand{\cT}{\mathbb{T}}
\newcommand{\wf}{\mathrm{WF}}
\newcommand{\data}{v}

\allowdisplaybreaks

\title{Analysis of the Linearized Problem of Quantitative Photoacoustic Tomography}

\author{Markus Haltmeier\footnotemark[2]  \and Lukas Neumann\footnotemark[3]  \and Linh V. Nguyen\footnotemark[4] \and Simon Rabanser\footnotemark[2]}

\date{\small \footnotemark[2] Department of Mathematics, University of Innsbruck\\
Technikerstrasse 13, A-6020 Innsbruck, Austria\\
  { \tt \{Markus.Haltmeier,Simon.Rabanser\}@uibk.ac.at}\\[1em]
\small\footnotemark[3] Institute of Basic Sciences in Engineering Science, University of Innsbruck\\Technikerstra{\ss}e 13, A-6020 Innsbruck, { \tt Lukas.Neumann@uibk.ac.at}\\[1em]
\small \footnotemark[4]
Department of Mathematics, University of Idaho\\
875 Perimeter Dr, Moscow, ID 83844, {\tt lnguyen@uidaho.edu}}

\begin{document}
\maketitle

\begin{abstract}
Quantitative image reconstruction in  photoacoustic tomography
requires the solution of a coupled physics inverse problem involving
light transport  and acoustic wave propagation.  In this paper we address this issue employing
the  radiative transfer equation  as accurate model for light transport. As  main theoretical results,
we derive  several  stability and uniqueness results for the linearized  inverse problem.  We consider  the case of   single  illumination as well as the case of multiple  illuminations assuming full or partial data.
The numerical solution of the linearized problem is much less costly than the solution of the non-linear problem. We present numerical simulations supporting  the stability results for the linearized problem and demonstrate that the linearized problem already gives accurate quantitative results.


\end{abstract}

\section{Introduction}
\label{sec:intro}

Photoacoustic tomography (PAT)  is a  recently developed coupled physics  imaging
modality  that combines the  high spatial resolution of ultrasound imaging
with the high contrast of optical imaging \cite{Bea11,KruKopAisReiKruKis00,WanAna11,Wan09b,XuWan06}. When a semitransparent sample is illuminated with a short pulse of laser light, parts of the optical energy are absorbed inside the sample, which in turn induces an acoustic pressure wave. In PAT, acoustic pressure waves are measured outside of the  object of interest and mathematical algorithms are used  to recover an image of the interior. Initial work and also  recent works in PAT concentrated on the  problem of reconstructing  the initial pressure distribution, which has been considered as final diagnostic image (see, for example,  \cite{AgrKucKun09,BurBauGruHalPal07,FilKunSey14,FinHalRak07,FinRak07,FinPatRak04,Hal14,haltmeier2016iterative,HalSchuSch05,huang2013full,Kow14,kuchment2014radon,Kun07a,KucKun08,nguyen2016dissipative,RosNtzRaz13,SteUhl09,XuWan06}). However, the initial pressure distribution only provides qualitative information about the tissue-relevant parameters.
This is due to the fact that the initial pressure distribution is the product of the optical absorption coefficient
and the spatially varying optical intensity which again indirectly depends on the tissue parameters.  Quantitative photoacoustic tomography (qPAT)  addresses this issue and aims at quantitatively estimating the tissue parameters by supplementing the  inversion of the acoustic wave equation
with an inverse problem for the light propagation  (see, for example, \cite{KruLiuFanApp95,CoxArrKoeBea06,CoxArrBea07a,RosRazNtz10,YaoSunHua10,AmmBosJugKang11,BalJolJug10,BalRen11,CheYa12,cox2012quantitative,TarCoxKaiArr12,RenHaoZha13,SarTarCoxArr,MamRen14,NatSch14,DezDez15,haltmeier2015single}).

The radiative transfer equation (RTE) is commonly considered
as a very accurate model for light transport in tissue (see, for example, see \cite{Arr99,DaLiVol6,EggSch15,Kan10}) and will be employed in this paper.
As proposed in \cite{haltmeier2015single} we work with a single-stage
reconstruction procedure for qPAT, where the optical parameters are reconstructed directly from the measured acoustical data. This is in contrast to the more common two-stage procedure, where the measured boundary pressure values are used to recover the initial internal pressure distribution in an intermediate  step, and the spatially varying  tissue parameters are estimated from the initial pressure distribution in a second step. However, as pointed out in \cite{gao2015limited,haltmeier2015single}, the two-stage approach has  several drawbacks, such as the missing capability of dealing with multiple illuminations
using incomplete acoustic measurements in each experiment. The single-stage strategy can also be combined with the diffusion approximation; see \cite{bal2016photo,bergounioux2014optimal,ding2015one,gao2015limited,yuan2012calibration}. The diffusion approximation is numerically less costly to solve than the RTE.  However, we work with the RTE since it is the more accurate model for light propagation in tissue.

\subsection{The linearized inverse problem of qPAT}

In this paper, we study the linearized inverse problem of qPAT using  the RTE. We present a uniqueness and stability analysis for both single and multiple illuminations.
Our strategy is to first analyze the optical (heating) and acoustic processes separately,  and then combine them together.
The uniqueness and stability analysis of the acoustic process is well established and we make use of existing  results. The study of the heating process, on the other hand, is much less understood.
Its analysis is our main emphasis and our contribution includes several
uniqueness and stability results.
We derive stability estimates of the form
\begin{equation*}
    C_1 \snorm{h_a}_{L^2(\Om)} \leq \snorm{\Wo\Do h_a}_{L^2(\Lambda\times (0,T))} \leq C_2 \snorm{h_a}_{L^2(\Om)}
      \,,
 \end{equation*}
for the unknown absorption parameter perturbation
$h_a= \mua -\mua^\lin$, where $\mua$ and $\mua^\lin$ are the actual
and background  absorption coefficients, respectively.
Here $\Wo \Do \colon L^2(\Om) \to L^2(\Lambda\times (0,T))$ is the linearized  forward operator of qPAT with respect to the attenuation at
$\mua^\lin$.
Such results are derived under different conditions
for vanishing scattering (Theorems~\ref{thm:sigma0-stablity}  and~\ref{thm:no}),
non-vanishing scattering (Theorem~\ref{thm:scatt}), as well as multiple
illuminations (Theorem~ \ref{thm:multiple}). Our analysis is inspired by~\cite{kuchment2012stabilizing} where microlocal analysis was employed to analyze the stability of an inverse problem with internal data (see also \cite{bal2013linearized,bal2014local,bal2014hybrid,kuchment2015stabilizing,montalto2013stability,widlak2015stability} for related works). We also take advantage of previous works on RTE  \cite{stefanov2003optical,EggSch14b} and microlocal acoustic wave inversion \cite{SteUhl09}.

The feasibility of solving the linearized  problem  is illustrated by numerical simulations presented
in Section~\ref{sec:num}. For  numerically solving the linearized problem  we use the Landweber iteration that can also be employed for the fully non-linear problem, see
\cite{haltmeier2015single}.
We point out that solving the linearized inverse  problem is  computationally much
less costly than solving the fully non-linear problem. Our numerical results  show that in  many situation    the solution of the linearized problem  already gives quite
accurate reconstruction results.

\subsection{Outline}

 The rest of the paper is organized as follows. In Section \ref{sec:forward} we
 present a Hilbert space framework for qPAT using the RTE. We consider
 the single illumination as well as the multiple illumination case and
 recall the well-posedness of the  non-linear forward operator. We
 further give  its G\^ateaux-derivative, whose inversion constitutes
 the linearized inverse problem of qPAT. In Section~\ref{sec:stability} we present our main
 stability and  uniqueness results for the linearized  inverse problem.
 Our theoretical results are supplemented  by numerical examples presented  in
 Section~\ref{sec:num}.  The paper concludes with a short summary and  outlook
 presented in  Section~\ref{sec:conclusion}.

\section{The forward problem in qPAT}
\label{sec:forward}

In this section we describe the forward problem of qPAT in a Hilbert space framework
using the RTE.
Allowing for $N$ different illumination  patterns, the forward problem is given by
a nonlinear operator $\Fo = (\Wo_i \circ \Ho_i)_{i=1}^N$, where
the operator $\Ho_i$ models the  optical heating  and $\Wo_i$ the acoustic measurement
due to the $i$-th illumination.  These operators are  described and analyzed
in detail in the  following.

\subsection{Notation}

Throughout this paper,   $\Om \subseteq \R^d$ denotes a convex bounded domain with Lipschitz-boundary $\partial\Om$,  where $d \geq 2 $ denotes the spatial dimension. We
write
\begin{align*}
\Gamma_-& \coloneqq
\left\{(x,\theta)\in\partial\Om\times \sph^{d-1}\mid \ip{\nu(x)}{\theta}<0\right\}   \,.
\end{align*}
Here  $\nu(x)$ denotes the outward unit normal at $x \in \partial \Om$ and
$\ip{a}{b}$ the standard scalar product of $a,b \in\R^d$.
We denote by  $L^2 \skl{\Gamma_-, \sabs{\ip{\nu}{\theta}}}$ the space of all measurable functions $\qo$ defined on $\Gamma_-$  such that
\begin{equation} \label{eq:W}
\norm{\qo}_{L^2\skl{\Gamma_-, \sabs{\ip{\nu}{\theta}}}}
 \coloneqq
  \sqrt {\int_{\Gamma_-}   \abs{\ip{\nu(x)}{\theta}} \abs{\qo(x,\theta)}^2 \rmd (x,\theta)} < \infty \,.
  \end{equation}
We further denote by $W^2(\Om\times \sph^{d-1})$  the space of all measurable functions defined on
$\Om\times \sph^{d-1}$  such that
\begin{equation}\label{eq:W2}
\norm{\Phi}_{W^2\skl{\Om \times \sph^{d-1}}}^2
\coloneqq
\norm{\Phi}_{L^2\skl{\Om \times \sph^{d-1}}}^2 +
\norm{\ip{\theta}{\nabla_x \Phi}}_{L^2\skl{\Om \times \sph^{d-1}}}^p +
\norm{\Phi|_{\Gamma_-}}_{L^2\kl{\Gamma_-,
\abs{\ip{\nu}{\theta}}}}^2
\end{equation}
is well defined and finite.
The subspace of all elements $\Phi \in W^2(\Om\times \sph^{d-1})$  satisfying
 $\Phi|_{\Gamma_-}=0$ will be denoted by  $W^2_0(\Om\times \sph^{d-1})$.
The spaces $L^\infty \skl{\Gamma_-, \sabs{\ip{\nu}{\theta}}}$, $W^\infty(\Om\times \sph^{d-1})$ and $W^\infty_0 (\Om\times \sph^{d-1})$ are defined in a similar manner by considering the  $L^\infty$-norms instead of the $L^2$-norms
in~\eqref{eq:W}, \eqref{eq:W2}.

For fixed positive numbers $\barmua, \barmus>0$ we write
\begin{equation} \label{eq:d2}
\dom  \coloneqq \set{ (\mua, \mus) \in  L^2 \kl{\Om} \times L^2 \skl{\Om } \mid  0 \leq \mus \leq  \barmus
\text{ and } 0 \leq \mua \leq  \barmua }  \,,
\end{equation}
for the parameter set of unknown attenuation and scattering coefficients.
Note that $\dom$ is a closed, bounded and convex subset of  $L^2 \kl{\Om} \times L^2 \kl{\Om}$ with empty interior.
Finally, we denote by $C^\infty_\Om(\R^d)$ and $L^p_\Om(\R^d)$, for $p \in  [1, \infty]$,
the set of  all elements in $C^\infty(\R^d)$ and $L^p(\R^d)$ that vanish outside of $\Omega$.

\subsection{The heating operator}

Throughout this subsection,  let $\qi\in L^\infty\skl{\Om\times \sph^{d-1}}$ and $\qo \in L^\infty\kl{\Gamma_-, \abs{\ip{\nu}{\theta}}}$ be a given source pattern and boundary light source, respectively. Further,  we denote by $k \colon \sph^{d-1} \times \sph^{d-1} \to \R$
the scattering kernel, which is supposed to be a symmetric and nonnegative function that satisfies
$\int_{\sph^{d-1}}  k \kl{\edot, \theta'} \rmd \theta' = 1$.

We model the optical radiation by a function $\Phi \colon  \Om \times \sph^{d-1} \to \R$, where $\Phi \kl{x, \theta}$ is the density of photons at location $x \in \Om$ propagating in direction $\theta \in \sph^{d-1}$.
The photon density is supposed to satisfy the RTE,
\begin{equation}\label{eq:strong}
\left\{
\begin{aligned}
\kl{ \ip{\theta}{\nabla_x}
+ \kl{\mua   + \mus - \mus \Ko} }
\Phi
&=
\qi
&& \text{ in } \Om \times \sph^{d-1}\\
\Phi|_{\Gamma_-}
&=
\qo
&& \text{ on } \Gamma_-\,.
\end{aligned}
\right.
\end{equation}
Here $\Ko \colon  L^2\skl{\Om \times \sph^{d-1}} \to L^2\skl{\Om \times \sph^{d-1}} $ denotes the  scattering operator defined  by
$\kl{\Ko \Phi} \kl{x, \theta} = \int_{\sph^{d-1}}  k(\theta, \theta' ) \Phi(x, \theta') \rmd \theta'$. See, for example, \cite{arridge2009optical,dorn1998transport,EggSch15,mcdowall2010stability,stefanov2003optical} for the RTE in optical tomography.

\begin{lemma}[Well-posedness of the RTE]  \label{lem:exist}
For every $(\mua,\mus) \in \dom$,  \eqref{eq:strong} admits a unique solution  $\Phi\in W^2(\Om\times \sph^{d-1})$. Moreover, there exists a
constant $C_2(\barmua, \barmus)$ only depending on $\barmua$ and $\barmus$, such that the following a-priori estimate holds
\begin{equation} \label{eq:apriori}
 \norm{\Phi}_{W^2\skl{\Om \times \sph^{d-1}}}
\leq
 C_2(\barmua, \barmus) \kl{ \norm{q}_{L^2\skl{\Om\times \sph^{d-1}}} + \norm{f}_{L^2\kl{\Gamma_-, \abs{\ip{\nu}{\theta}}}} }\,.
\end{equation}
\end{lemma}

\begin{proof}
See \cite{EggSch14b}.
\end{proof}

The absorption of photons causes a non-uniform heating of the tissue that is proportional to the total amount of absorbed photons.
We model this by the heating operator $\Ho_{\qo,\qi} \colon \dom\to L^2(\Om)$ that is defined by
\begin{equation}\label{eq:operator}
\Ho_{\qo,\qi} \kl{\mua,\mus}(x)
\coloneqq
   \mua(x)  \int_{\sph^{d-1}} \Phi(x, \theta) \rmd \theta
 \quad \text{ for } x \in \Om  \,,
\end{equation}
with  $\Phi$ denoting the  solution of~\eqref{eq:strong}.

\begin{lemma}\label{lem:contH}
The heating operator $\Ho_{\qo,\qi} \colon \dom\to L^2(\Om)$ is  well  defined, Lipschitz-continuous and weakly continuous.
\end{lemma}

\begin{proof}
See~\cite{haltmeier2015single}.
\end{proof}

We next compute the derivative of $\Ho_{\qo,\qi}$.
For that purpose we call $h \in  L^2(\Om) \times L^2(\Om)$ a feasible direction at $\mu = (\mua, \mus) \in \dom$ if there exists some $\eps >0$ such that $\mu  +  \eps h \in \dom$.
The set of all feasible directions at $\mu$ will be denotes by $\dom\skl{\mu}$.
For $\mu \in \dom$ and $h\in \dom\skl{\mu}$ we denote the  one-sided  directional derivative of $\Ho_{\qo,\qi}$ at $\mu$ in direction $h$ by $\Ho_{\qo,\qi}' (\mu)(h)$.
If   $\Ho_{\qo,\qi}' (\mu)(h)$ and $- \Ho_{\qo,\qi}' (\mu)(-h)$ exist and coincide on a dense subset of directions and $h \mapsto \Ho_{\qo,\qi}' (\mu )(h)$ is bounded and linear, we say that $\Ho$ is G\^ataux differentiable at $\mu $ and call $\Ho_{\qo,\qi}' (\mu)$
the G\^ataux derivative of $\Ho_{\qo,\qi}$ at $\mu$.

\begin{proposition}[Differentiability of $\Ho_{\qo,\qi}$]\label{prop:Hdiff}
Suppose $\mu = \skl{\mua, \mus} \in \dom$.
\begin{enumerate}
\item The one-sided directional derivative of $\Ho_{\qo,\qi}$ at  $\mu$ in
 any feasible direction $h = (h_a, h_s) \in   \dom\skl{\mu}$
 exists and is given by
\begin{equation} \label{eq:derH}
\Ho_{\qo,\qi}' \skl{\mu}(h)
= h_a \int_{\sph^{d-1}}\Phi(\edot,\theta)\rmd  \theta
-
\mua\int_{\sph^{d-1}} \Psi(\edot,\theta)\rmd  \theta \,.
\end{equation}
Here $\Phi$ is the  solution of \eqref{eq:strong} and
$\Psi$  satisfies the  RTE
\begin{equation}\label{eq:der}
\left\{
\begin{aligned}
\kl{ \ip{\theta}{\nabla_x}
+
\mua + \mus - \mus\Ko } \Psi
&= \kl{h_a + h_s - h_s\Ko} \Phi
&& \text{in } \Om \times \sph^{d-1}\\
\Psi|_{\Gamma_-}
&=0
&& \text{on } \Gamma_-\,.
\end{aligned}
\right.
\end{equation}

\item If $\mua, \mus > 0$, then $\Ho_{\qo,\qi}$ is G\^ateaux differentiable at $\skl{\mua, \mus}$.
\end{enumerate}
\end{proposition}

\begin{proof}
See \cite{haltmeier2015single}.
\end{proof}

\subsection{The wave equation}
\label{sec:wave}

The optical heating  induces an acoustic pressure wave
$p \colon \R^d \times \kl{0, \infty} \to \R$, which satisfies the initial value problem
\begin{equation} \label{eq:wave-fwd}
	\left\{ \begin{aligned}
	 \partial_t^2 p (x,t) - \Delta p(x,t)
	&=
	0 \,,
	 && \text{ for }
	\kl{x,t} \in
	\R^d \times \kl{0, \infty}
	\\
	p\kl{x,0}
	&=
	h(x) \,,
	&& \text{ for }
	x  \in \R^d
	\\
	\partial_t
	p\kl{x,0}
	&=0 \,,
	&& \text{ for }
	x  \in \R^d \,.
\end{aligned} \right.
\end{equation}
For the sake of simplicity in \eqref{eq:wave-fwd} and below  we assume the speed of sound to be constant and rescaled to one. Further, the initial data  $h$ is assumed to be supported in $\Om$.

We suppose that the acoustic measurements are made on a subset $\La \subseteq \mS$, where $D \subseteq \R^d$  is a bounded domain with smooth boundary  such that $D \supseteq  \bar \Om $. The acoustic forward operator corresponding to the measurement set $\La$ is
defined by
\begin{align}\label{eq:waveoperator}
	&\Wo_{\Lambda,T}  \colon
	C_\Om (\R^d)
	\subseteq L^2_\Om (\R^d) \to L^2(\Lambda \times (0, T))
	\colon
	h \mapsto p|_{\Lambda \times (0, T)}
	\,,
\end{align}
where $p \colon  \R^d \times (0, T) \to \R$
denotes the unique solution of \eqref{eq:wave-fwd}.

\begin{lemma} \label{lem:contW}
$\Wo_{\La,T} $ is well  defined and bounded and therefore can be uniquely extended to a bounded linear
operator $\Wo_{\La,T} \colon L^2_\Om (\R^d) \to  L^2\kl{\Lambda \times (0, T)}$.
\end{lemma}

\begin{proof}
See for example \cite{haltmeier2015single}.
\end{proof}

The operator  $\Wo_{\La,T}$ can be evaluated by well-known
solution formulas (see, for example, \cite{Eva98,Joh82}). In
 two spatial dimensions a solution is given by
\begin{equation} \label{eq:sol}
\kl{\Wo_{\La,T} h}
\kl{y,t} =
      \frac{1}{2\pi}
     \frac{\partial}{\partial t}
     \int_0^t
     \int_{\sph^1} h \skl{y +  t \omega}
     \frac{r}{\sqrt{t^2-r^2}}
      \, \rmd \omega \, \rmd r \,,
\end{equation}
Because $\Wo_{\La,T}$ is bounded and linear,
the  adjoint $\Wo_{\La,T}^\ast \colon L^2\kl{\Lambda \times (0, T)} \to L^2_\Om (\R^d)$  is again well defined
and bounded. Explicit expressions of $\Wo_{\La,T}^\ast $  are easily
deduced from explicit expression  for the solution of the wave equation.
For example, in two spatial dimensions we have
\begin{equation} \label{eq:adjoint}
    \kl{\Wo_{\La, T}^* v}\kl{x}
    =
    -\frac{1}{2\pi}
    \int_{\La}
    \int_{\sabs{x-y}}^T
    \frac{\partial_t  v \kl{y,t}}
    {\sqrt{t^2 - \sabs{x-y}^2}}   \, \rmd t  \, \rmd S(y) \,,
\end{equation}
for every $v \in C_c^1(\Lambda \times (0, T) )$.

\subsection{The forward operator in qPAT}
\label{sec:forward-multiple}

As being common in qPAT, we are interested in the case of a  single illumination as well as
multiple illuminations.  Suppose that $\qi_i \in L^\infty\skl{\Om \times \sph^{d-1}}$ and  $\qo_i \in L^\infty\kl{\Gamma_- , \sabs{\ip{\nu}{\theta}}}$,
for $i = 1, \dots , N$, denote given source patterns and boundary light sources, respectively, where
$N$ is the number of different  illumination patterns. The case $N=1$ corresponds to a
single illumination. We further denote by $\La_i$ the surface where the acoustic measurements
 for the $i$-th illumination are made and $T_i$ the measurement duration.
 Let us mention that multiple illuminations are often implemented in PAT
  (see, e.g.,~\cite{ermilov2016three} for such an experimental setup).
 Single-stage qPAT with multiple illuminations has been studied in \cite{ding2015one,gao2015limited,haltmeier2015single,bal2016photo,lou2016impact}. Moreover, multiple illuminations have been proposed in several mathematical works to stabilize various inverse problems with internal data (e.g., \cite{bal2011quantitative,BalRen11,bal2013linearized,bal2014local,bal2014hybrid,kuchment2015stabilizing,montalto2013stability,widlak2015stability}).

For each illumination and measurement, we denote by
$\Ho_i  \coloneqq \Ho_{\qo_i, \qi_i}$ and $\Wo_i  \coloneqq  \Wo_{\La_i,T}$ the corresponding heating and acoustic operator and by $\Fo_i = \Wo_i \circ \Ho_i$ the resulting forward operator.
The total forward operator in qPAT allowing multiple illuminations is  then given by
\begin{equation}
\Fo = (\Fo_i)_{i=1}^N \colon
\dom \to  \kl{ L^2(\Lambda \times (0, T)) }^N
\end{equation}
From Lemmas~\ref{lem:contH}  and~\ref{lem:contW} it follows that the
forward operator   $\Fo$ is  well  defined, Lipschitz-continuous and weakly continuous.

\begin{proposition}[Differentiability of $\Fo$]\label{prop:Mdiff}
Let $\mu = (\mua, \mus) \in \dom$.
\begin{enumerate}
\item The one-sided directional derivative of $\Fo$ at  $\mu$ in
 any feasible direction $h = (h_a, h_s) \in   \dom\skl{\mu}$
 exists and, with $\Ho_i'$ as in Proposition~\ref{prop:Hdiff},
 is given by  \begin{equation} \label{eq:Mdiff}
	\Fo'(\mu)
	=  (\Wo_i \circ \Ho_i'(\mu))_{i=1}^N \,.
\end{equation}

\item If $\mua, \mus > 0$, then $\Fo$ is G\^ateaux differentiable at $\mu$.
\end{enumerate}
\end{proposition}

\begin{proof}
Follows from Proposition~\ref{prop:Hdiff} and
Lemma \ref{lem:contW}.
\end{proof}

The derivative $\Fo'(\mu)$ is the linearized forward operator in qPAT
that  we will analyze in the following.

\section{Analysis of the linearized inverse problem}
\label{sec:stability}

In this section we study  uniqueness and stability of the problem of inverting
$\Fo'(\mua,\mus)$, where  $(\mua^\lin,\mus^\lin) \in \dom$
is a fixed pair of background optical absorption and scattering coefficients.

We denote by $\Vo \colon  W^\infty(\Om\times \sph^{d-1})
\to L^\infty(\Om \times \sph^{d-1})$ the transport operator defined  by
$\Vo\Phi  \coloneqq  \kl{ \ip{\theta}{\nabla_x} + \mua} \Phi$ and
write $\Vo_0$ for the restriction to $W^\infty_0(\Om\times \sph^{d-1})$.
Then, $\Vo_0$ is invertible and its inverse is given by
\begin{equation} \label{eq:inv-V}
(\Vo_0^{-1} \Psi)(x,\theta) = \int_0^{\ell(x,\theta)} e^{-\int_0^t \mua(x- \tau \theta) \, \rmd \tau}\, \Psi(x-t \theta,\theta) \, \rmd t \,.
\end{equation}
Here, $\ell(x,\theta)$ is the supremum over all $s>0$ such that $x-s\theta \in \Om$. That is, $\ell(x,\theta)$ is the distance from $x$ to the boundary $\partial \Om$ along the direction $-\theta$.
It is easy to see that $\Vo_0^{-1}$ is a bounded operator when considered as mapping
from $L^\infty(\Om \times \sph^{d-1})$ into itself and  satisfies
\begin{equation}\label{E:norm}
	\snorm{\Vo_0^{-1}}_{L^\infty, L^\infty} \leq \diam(\Om) \,.
\end{equation}

\subsection{An auxiliary result}

The following result plays a key role in our
subsequent analysis.

\begin{proposition}\label{prop:PDO}
Suppose $a  \in C^\infty\skl{\Om\times \sph^{d-1} \times [0,\infty)}$ is compactly supported with respect to the last variable $t$. For each $q \in C_c^\infty(\Om)$, we define
\begin{equation*}
P(\qi)(x)  \coloneqq  \int_0^\infty \int_{\sph^{d-1}} a(x,\theta,t) \, \qi(x- t\theta)
\rmd \theta \,  \rmd t
\end{equation*}
Then, $P$ extends to a pseudodifferential operator of order at most $-1/2$ on $\Om$.
\end{proposition}

\begin{proof} We have
 $\qi(x- t \theta) = (2 \pi)^{-d} \int_{\R^d} \int_{\R^d} e^{\imi\left<x-t \theta-y,\xi\right>} \qi(y) \, \rmd y \, \rmd \xi$.  Therefore,
\begin{align*}
P(\qi)(x)
&=\int_0^\infty \int_{\sph^{d-1}} a(x,\theta,t) \, \qi(x- t\theta) \rmd \theta \,  \rmd t\\
&=\frac{1}{(2 \pi)^d}\int_0^\infty \int_{\sph^{d-1}} a(x,\theta,t) \, \int_{\R^d} \int_{\R^d} e^{\imi \left<x-t \theta-y,\xi\right>} \qi(y) \, \rmd y \, \rmd \xi \rmd \theta \,  \rmd t\\
&= \frac{1}{(2 \pi)^d} \int_{\R^d}\int_{\R^d} b(x,\xi) e^{\imi\left<x-y,\xi\right>} \qi(y)
\, \rmd y \, \rmd \xi \,,
\end{align*}
where
$
b(x,\xi) \coloneqq \int_{\sph^{d-1}}  \int_0^\infty  a(x,\theta,t) e^{- \imi t \sinner{\theta}{\xi}}  \, \rmd t \, \rmd \theta \,.
$

We prove  that $(x,\xi) \mapsto b(x,\xi)$ is a symbol of order at most $-1/2$.
For that purpose, let $0 \leq \chi \in C^\infty_0(\mathbb{R}) \leq 1$ be a cut-off function that is equal to $1$ on $[-1,1]$ and zero outside of $[-2,2]$. Then
\begin{multline} \label{eq:I12}
b(x,\xi) = \int_{\sph^{d-1}}  \int_0^\infty \left(1-\chi \Big(\frac{\sinner{\xi}{\theta}}{\sabs{\xi}^{1/2}} \Big) \right) \,  a(x,\theta,t) e^{- \imi t \sinner{\theta}{\xi}}  \, \rmd t \, \rmd \theta
\\ + \int_{\sph^{d-1}}  \int_0^\infty \chi \Big(\frac{\sinner{\xi}{\theta}}{\sabs{\xi}^{1/2}}\Big)  \,  a(x,\theta,t) e^{- \imi t \sinner{\theta}{\xi}}  \, \rmd t \, \rmd \theta  =: I_1(x,\xi) + I_2(x,\xi).
\end{multline}
To estimate  $I_1(x,\xi)$, we write
\begin{align*}
I_1(x,\xi)  &=
\int_{\sph^{d-1}}  \int_0^\infty \left(1-\chi \left(\frac{\sinner{\xi}{\theta}}{\sabs{\xi}^{1/2}}\right) \right) \,  a(x,\theta,t) e^{- \imi t \sinner{\theta}{\xi}}  \, \rmd t \, \rmd \theta
\\ &=
\int_{\sph^{d-1}}  \left(1-\chi \left(\frac{\sinner{\xi}{\theta}}{\sabs{\xi}^{1/2}}\right) \right) \,\frac{1}{\imi \sinner{\theta}{\xi} } \Bigl[ a(x,\theta,0)  +  \int_0^\infty a'(x,\theta,t) e^{- \imi t \sinner{\theta}{\xi}}  \, \rmd t  \Bigr] \, \rmd \theta
\\ &=
\int_{\sph^{d-1}, |\sinner{\xi}{\theta}| \geq \sabs{\xi}^{1/2}}
\frac{1-\chi\bigl(\frac{\sinner{\xi}{\theta}}{\sabs{\xi}^{1/2}}\bigr) }{\imi \sinner{\theta}{\xi} } \Bigl[ a(x,\theta,0)  +  \int_0^\infty a'(x,\theta,t) e^{- \imi t \sinner{\theta}{\xi}}  \, \rmd t \Bigr] \, \rmd \theta \,,
\end{align*}
and therefore  $ \abs{I_1(x,\xi)} \leq C \sabs{\xi}^{-1/2}$ for   $\sabs{\xi} \geq 1$.
To estimate $I_2(x,\xi)$, note that  the set $\{\theta \in \sph^{d-1} \mid  \sabs{\sinner{\xi}{\theta}} \leq 2 \sabs{\xi}^{1/2}\}$ has measure proportional
to  $\sabs{\xi}^{-1/2}$ which shows
$$\abs{I_2(x,\xi)} = \Bigl|\int_{\sph^{d-1}}  \int_0^\infty
\chi\kl{\frac{\sinner{\xi}{\theta}}{\sabs{\xi}^{1/2}}}  \,  a(x,\theta,t) e^{- \imi t \sinner{\theta}{\xi}}  \, \rmd t \, \rmd \theta \Bigr|  \leq  C\, \sabs{\xi}^{-\frac{1}{2}}  \; \text{ for   }
\sabs{\xi} \geq 1  \,.$$
Together with \eqref{eq:I12}, we  obtain
$|b(x,\xi)|  \leq C \sabs{\xi}^{-\frac{1}{2}}$ for   $\sabs{\xi} \geq 1$.

\medskip
Next we estimate $\partial_{\xi_j} b(x,\xi) = -\imi\,  \int_{\sph^{d-1}}  \int_0^\infty  [t \theta_j a(x,\theta,t) ] e^{- \imi t \sinner{\theta}{\xi}}  \, \rmd t \, \rmd \theta$.  We consider the case  $d=2$ only, the proof for general dimension   follows in a similar manner.
Let us write $\theta = (\cos (\phi), \sin (\phi))$. Then, using  the expression of the gradient operator in polar coordinates, we obtain
\begin{equation*}
 \biggl[\xi_1 \kl{ \cos (\phi) \partial_t - \frac{\sin (\phi)}{t}  \partial_\phi  } +
 \xi_2 \kl{ \sin (\phi) \partial_t + \frac{\cos (\phi)}{t} \partial_\phi }
 \biggr]\, e^{- \imi t \sinner{\theta}{\xi}}
 = -\imi \, \sabs{\xi}^2 \, e^{- \imi t \sinner{\theta}{\xi}}
\end{equation*}
Together with one integration by parts this shows
\begin{align*}
\partial_{\xi_j} b(x,\xi)
&=
\frac{1}{\sabs{\xi}^2}  \,
\int_{\sph^{d-1}}  \int_0^\infty  t \theta_j a(x,\theta,t)
 \, \biggl[
 \xi_1 \kl{ \cos (\phi) \partial_t - \frac{\sin (\phi)}{t}  \partial_\phi  }
 \\
 & \hspace{0.25\textwidth}+
 \xi_2 \kl{ \sin (\phi) \partial_t + \frac{\cos (\phi)}{t} \partial_\phi }  \biggr] \, e^{- \imi t \sinner{\theta}{\xi}}  \, \rmd t \, \rmd \theta
 \\
  &=-\frac{1}{\sabs{\xi}^2}  \,  \int_{\sph^{d-1}}  \int_0^\infty
  e^{- \imi t \sinner{\theta}{\xi}}
  \biggl[
  \xi_1 \kl{ \cos (\phi) \partial_t - \frac{\sin (\phi)}{t}  \partial_\phi  }
   \\
 & \hspace{0.2\textwidth}+
  \xi_2 \kl{ \sin (\phi) \partial_t + \frac{\cos (\phi)}{t} \partial_\phi }
  \biggr] \, \kl{ t \theta_j a(x,\theta,t)  }   \, \rmd t \, \rmd \theta\,.
  \end{align*}
Using that the functions
$ (  \cos (\phi) \partial_t -\sin (\phi) t^{-1}  \partial_\phi  )(t \theta_j a) $,
$  (\sin (\phi) \partial_t +  \cos (\phi) t^{-1} \partial_\phi )(t \theta_j a)$ are contained $ C^\infty \skl{\Om \times \sph^{d-1} \times [0,\infty)}$ and repeating the argument above, we  conclude that
$\sabs{\partial_{\xi_j} b(x,\xi) } \leq C \sabs{\xi}^{-3/2}$.
Finally, in a similar manner one verifies $\sabs{\partial^\beta_{x} \partial^\alpha_{\xi} b(x,\xi) } \leq C_{\al,\beta} \sabs{\xi}^{-1/2+|\alpha|} $ for all  $\al, \beta  \in \N^d$, which
concludes the proof.
\end{proof}

\subsection{Vanishing scattering}

In this subsection we assume zero scattering, and consider a single illumination $N=1$.
For given  $\mua > 0 $ we write $\Do = \Ho'(\mua,0)( \edot,0)$ for the linearized partial forward operator. It is given by (see Proposition~\ref{prop:Hdiff})
\begin{equation} \label{eq:lin-0}
\Do(h_a)(x)
= h_a(x) \int_{\sph^{d-1}}\Phi(x,\theta)\rmd \, \theta - \mua(x) \, \int_{\sph^{d-1}} \Psi (x,\theta)\rmd  \theta \,,
\end{equation}
where $\Phi$ and $\Psi$ satisfy the following background and linearized problem
\begin{align}\label{eq:bac}
&  \Vo \Phi  = \qi
&&  \text{such that } \Phi|_{\Gamma_-} = \qo \,,
\\ \label{eq:lin}
& \Vo  \Psi  = \Vo_0  \Psi  = h_a \Phi
&&  \text{such that } \Psi|_{\Gamma_-} =0 \,,
\end{align}
respectively.   According to \eqref{eq:inv-V}, the solution of \eqref{eq:lin} equals
\begin{equation} \label{eq:Psi}
\Psi(x,\theta)  = \Vo_0^{-1}(h_a \Phi)(x,\theta)=\int_0^{\ell(x,\theta)} (h_a \Phi)(x-t \theta,\theta) e^{- \int_0^t \mua(x- \tau \theta) \rmd \tau} \rmd t \,.
\end{equation}
For notational conveniences, we denote by $\fluence \coloneqq \int_{\sph^{d-1}} \Phi(\edot,\theta) \, \rmd \theta$ the background fluence.  Further, let us fix a domain  $\Om_0 \Subset \Om$. We will always assume that $\supp(h_a) \subseteq \Om_0$.

We will also study the linearized forward operator $\Fo'(\mua,0)(\edot,0) = \Wo \Do$,
 where $\Wo = \Wo_{\La, T} $ is the
solution operator of the wave equation for the measurement set $\La \subseteq \partial D$
and measurement time $T>0$; see Subsection~\ref{sec:wave}.
For that purpose we recall the visibility condition \ref{A:Visible} for the wave inversion problem,
which states that  any line through $x \in \Om$ intersects  $\Lambda$ at a point of distance
less than $T$ from $x$.

\begin{lemma} \label{lem:psi} If $\Phi \in C^\infty(\overline \Om \times \sph^{d-1})$, then   $h_a \to \int_{\sph^{d-1}} \Psi (\edot,\theta)\rmd  \theta $
is a pseudo-differential operator of order at most $-1/2$ on $\Om_0$.
\end{lemma}

\begin{proof}
Let  $\varphi \in C_0^\infty(\Om)$ be such that $\varphi = 1$
in $\Om_0$. Then $h\varphi=h$, and according to \eqref{eq:Psi} we have
$$\int_{\sph^{d-1}} \Psi(x,\theta) \,\rmd \theta =  \int_0^\infty \int_{\sph^{d-1}} a(x,\theta,t) \, h_a(x- t\theta) \rmd t \, \rmd \theta,$$
with  $a(x,\theta,t)  \coloneqq  \varphi(x- t \theta) \, \Phi(x-t \theta,\theta) \, e^{- \int_0^t \mua(x- \tau \theta) \rmd \tau}$.
Together with  Proposition~\ref{prop:PDO}  this implies that
$h_a \to \int_{\sph^{d-1}} \Psi(\edot,\theta) \, \rmd \theta$ is a pseudo-differential operator
of order at most $-1/2$ and concludes the proof.
\end{proof}

Lemma~\ref{lem:psi} in particular implies that $h_a \to \int_{\sph^{d-1}} \Psi (x,\theta)\rmd  \theta$ is boundedly maps $L^2 \skl{\Om}$ to the Sobolev space $H^{1/2} \skl{\Om}$. Such  weaker result also follows from  the averaging  lemma,
which states that the averaging operator
$W^2 \skl{\Om \times \sph^{d-1}} \to H^{1/2} \skl{\Om}
\colon F \mapsto  \int_{\sph^{d-1}} F(\edot ,\theta)  \rmd \theta$  is bounded~\cite{devore2001averaging,golse1988regularity,Mok97}.
From the stronger (localized) statement of Lemma~\ref{lem:psi}, we can infer that $h_a \mapsto \Do h_a  = h_a  \fluence  - \mua  \, \int_{\sph^{d-1}} \Psi (\edot,\theta)\rmd  \theta$ is a pseudo-differential operator of
order $0$ and principal symbol $\fluence$.
Moreover, if the background fluence $\fluence$ is positive on $\overline \Om_0$ one concludes the following.

\begin{theorem} \label{thm:no-scattering}
Suppose $\fluence>0 $ on $\overline \Om_0$.
Then  the following hold:
\begin{enumerate}
\item \label{thm:no-scattering0}
$\Do$ is an elliptic pseudodifferential operator of order zero and Fredholm.

\item \label{thm:no-scattering1}
$\wf(h) \cap \cT^* \Om_0 =  \wf( \Do (h))  \cap \cT^* \Om_0 $.

\item \label{thm:no-scattering2}
$\dim \kl{\ker (\Do) } <\infty$.

\item \label{thm:no-scattering3}
 $\ker( \Do)   \subseteq C^\infty(\Om)$.

\item \label{thm:no-scattering4}
If, additionally, the visibility condition~\ref{A:Visible} holds, then
\begin{itemize}
\item
$\wf( \Ao \, \chi_{\La, T} \Wo \Do (h))  \cap \cT^* \Om_0=\wf(h) \cap \cT^* \Om_0$
\item
$\dim \kl{\ker ( \Ao \, \chi_{\La, T} \Wo \Do) }<\infty$
\item
$\ker(\Ao \, \chi_{\La, T} \Wo \Do) \subseteq C^\infty(\Om) $
\end{itemize}
Here $\Ao$ is the time reversal operator (defined by  \eqref{eq:wave-back}),
 and $\chi_{\La,T} \in C^\infty(\mS \times [0,\infty))$ is a nonnegative function with
$\supp(\chi_{\La,T}) = \Lambda \times [0,T]$.
\end{enumerate}

\end{theorem}

\begin{proof}
\ref{thm:no-scattering0}-\ref{thm:no-scattering3}.
Because we have $\fluence>0 $ on $\overline \Om_0$, the operator $\Do$ is elliptic and Fredholm.
The ellipticity of $\Do$ implies $\wf(h) \cap \cT^* \Om_0 =  \wf(\Do  h )  \cap \cT^* \Om_0$ and  $\ker(\Do) \subseteq C^\infty(\Om)$ (see, e.g., \cite{hormander1971fourier,treves1980introduction}).  The Fredholm property implies $\dim \ker(\Do) < \infty$ and concludes the proof.

\ref{thm:no-scattering4}
From Theorem~\ref{T:SU}, we obtain $\Ao \chi_{\La,T} \Wo \Do$ is a pseudo-differential operator of order $0$ whose principal symbol is
$\tfrac{1}{2} \fluence(x)   \sum_{\sigma = \pm} \chi_{\La,T}(y_\sigma(x,\xi),t_\sigma(x,\xi))$, where $y_\pm(x,\xi) = \ell_\pm(x,\xi) \cap \mS$ and $t_\pm(x,\xi) = \abs{x- y_\pm(x,\xi)}$ are the location and time when the two singularities starting  at  $(x,\xi) \in \wf(h)$ hit the observation surface.
Under the Assumption~\ref{A:Visible},
$\Ao \chi_{\La,T} \Wo \Do$ is elliptic and Fredholm and concludes  the
proof.
\end{proof}

For our further analysis let us introduce the abbreviations
\begin{align} \label{eq:phimin}
\fluence_{\min}
& \coloneqq  \inf \set{ \fluence(x) \mid x \in \Om_0 }
\,,\\ \label{eq:ell}
\ell_+ (x) & \coloneqq  \frac{1}{\sabs{\sph^{d-1}}} \int_{\sph^{d-1}} \ell(x,\theta) \, d\theta  \quad \text{ for } x \in \Om
\,,\\
\ell_\infty(x) & \coloneqq
\max \{ \ell(x,\theta): \theta \in \sph^{d-1}\}  \quad \text{ for } x \in \Om
\,.
\end{align}
Recall that $\ell(x,\theta)$ is defined as the supremum over all $s>0$ such that
$x-s\theta \in \Om$.

\begin{lemma} \label{lem:sigma0-injective}
$\Do$ is injective on $L^\infty(\Om_0)$, provided that
\begin{equation} \label{E:Inj}
\sabs{\sph^{d-1}} \norm{\mua  \ell}_\infty  \norm{\Phi}_{L^\infty(\Om_0)}
< \fluence_{\min}  \,.
\end{equation}
\end{lemma}

\begin{proof}
Recall that $\Do(h_a) =   \fluence  h_a - \mua
\int_{\sph^{d-1}} \Psi (\edot,\theta)\rmd  \theta$.
In order to show the injectivity of $\Do$  it therefore
 suffices to prove that for any $h_a\neq 0$, we have
\begin{equation}\label{E:Ineq}
 \norm{\fluence \, h_a}_{L^\infty(\Om_0)} >
 \norm{\mua \, \int_{\sph^{d-1}} \Psi (\edot,\theta)\rmd  \theta }_{L^\infty(\Om_0)}.
\end{equation}
For that purpose,  the left hand side of inequality \eqref{E:Ineq} is estimated as
 $\norm{\fluence \, h_a}_{L^\infty(\Om_0)} \geq \snorm{h_a}_{L^\infty(\Om_0)}  \, \fluence_{\min}$. On the other hand, let us recall
\begin{equation*}
	\Psi(x,\theta)= \Vo_0^{-1} (h_a \Phi)(x,\theta)
	=  \int_0^{\ell(x,\theta)}  e^{- \int_0^t \mua(x- \tau \theta) \rmd \tau}  \,\Phi(x-t \theta,\theta) \,  h_a(x- t\theta) \rmd t \,.
\end{equation*}
Therefore, the right hand side of \eqref{E:Ineq} can be estimated as
\begin{multline*}
\Big| \mua(x) \, \int_{\sph^{d-1}}  \Psi(x,\theta)  \rmd \theta \Big|  =
 \Big| \mua(x) \, \int_{\sph^{d-1}}  \int_0^{\ell(x,\theta)}  e^{- \int_0^t \mua(x- \tau \theta) \rmd \tau}  \\  \times h_a(x- t\theta) \, \Phi(x- t\theta,\theta) \rmd t \rmd  \theta \Big|
 \leq  \sabs{\sph^{d-1}}   \norm{\mua \, \ell}_\infty  \norm{\Phi}_{L^\infty(\Om_0)} \snorm{h_a}_{L^\infty(\Om_0)} \,.
\end{multline*}
Together with  \eqref{E:Inj} this yields~\eqref{E:Ineq}.
\end{proof}

From Lemma \ref{lem:sigma0-injective} and  the Fredholm property of $\Do$
we conclude the following two-sided stability results for inverting
$\Do$ and $\Wo\Do$.

\begin{theorem} \label{thm:sigma0-stablity}
Suppose that  \eqref{E:Inj} is satisfied.

\begin{enumerate}
\item\label{thm:sigma0-stablity1}
There exist constants $C_1, C_2 > 0$ such that:
\begin{equation}\label{eq:sigma0-stablity1}
\forall h_a \in L^2(\Om_0) \colon \quad
 C_1\snorm{h_a}_{L^2(\Om_0)} \leq
 \|\Do(h_a)\|_{L^2(\Om_0)} \leq C_2 \snorm{h_a}_{L^2(\Om_0)}
 \,.
\end{equation}

\item\label{thm:sigma0-stablity2}
If  Condition~\ref{A:Visible} is satisfied, then for some constants $C_1', C_2' > 0$,
\begin{equation}\label{eq:sigma0-stablity2}
\forall h_a \in L^2(\Om_0) \colon \quad
C_1'\snorm{h_a}_{L^2(\Om)} \leq \|\Wo  \Do(h_a)\|_{L^2(\Lambda \times (0,T))} \leq C_2' \snorm{h_a}_{L(\Om)}\,.
\end{equation}
\end{enumerate}
\end{theorem}

\begin{proof}
\ref{thm:sigma0-stablity1}
Choose $\Om_1 \Supset \Om_0$ such that $\fluence>0$ on $\overline \Om_1$.
Assume that $h \in L^2(\Om_0)$ is such that $\Wo(h) = 0$.
Then, applying Theorem~\ref{lem:psi} for $\Om_1$ (instead of $\Om_0$), we obtain $h \in C^\infty(\Om_1)$, which implies $h \in L^\infty(\Om_0)$.
Now Lemma~\ref{lem:sigma0-injective} gives $h = 0$. Therefore $\Do$ is injective on
$L^2(\Om_0)$. Because $\Do$ is Fredholm, this gives~\eqref{eq:sigma0-stablity1}.

\ref{thm:sigma0-stablity2}
This follows from \ref{thm:sigma0-stablity1} and the stability of the wave equation.
\end{proof}

Condition~\eqref{E:Inj} may be quite strong  when the solution
$\Phi(x,\theta)$ varies a lot. This is especially relevant for the case of multiple
illumination.  In the following we  therefore provide a different condition
for the case that the  background problem is sourceless, that is $\qi=0$.
Note that this is not a severe restriction since in qPAT the optical illumination
is usually modeled by a boundary pattern $\qo$.

\begin{theorem} \label{thm:no}
 Suppose  $\|\mua \, \ell_\infty \|_{L^\infty(\Om_0)} <1$
and $\qi=0$.
\begin{enumerate}
\item\label{thm:sigma0-stablity0-no}
The operator $\Do$ is injective.

\item\label{thm:sigma0-stablity1-no}
There exists some  constant $C_1 > 0$ such that
\eqref{eq:sigma0-stablity1} holds.

\item\label{thm:sigma0-stablity2-no}
If  additionally Condition~\ref{A:Visible} is satisfied, then
\eqref{eq:sigma0-stablity2} holds  for some $C_2 > 0$.
\end{enumerate}
\end{theorem}

\begin{proof}
\ref{thm:sigma0-stablity0-no}
From \eqref{eq:Psi} we have $$| \Psi(x,\theta)| \leq \snorm{h_a}_{L^\infty(\Om_0)}  \int_0^\infty \Phi(x-t \theta,\theta) \, e^{- \int_0^t \mua(x- \tau \theta) \rmd \tau}  \rmd t.$$
Since $\Phi$ satisfies  $\Vo \Phi =0$, the function
$ \Phi(x-t \theta,\theta) \, e^{- \int_0^t \mua(x- \tau \theta) \rmd \tau} $ is independent of $t$. This implies  $\Phi(x-t \theta,\theta) \, e^{- \int_0^t \mua(x- \tau \theta) \rmd \tau}=  \Phi(x,\theta)$ and therefore   $| \Psi(x,\theta)| \leq \snorm{h_a}_{L^\infty(\Om_0)} \, \ell(x,\theta)  \, \Phi(x,\theta)$.
Hence,
$$ \int_{\sph^{d-1}}  |\Psi(x,\theta)| \, \rmd \theta  \leq \snorm{h_a}_{L^\infty(\Om_0)}  \,  \int_{\sph^{d-1}} \ell(x,\theta) \, \Phi(x,\theta)\, \rmd \theta \leq \snorm{h_a}_{L^\infty(\Om_0)}  \, \ell_\infty(x) \, \fluence(x).$$
Next recall  $\Do(h_a) =  \fluence \, h_a -  \mua \, \int_{\sph^{d-1}}  \Psi(\edot ,\theta) \, \rmd  \theta$. Therefore,
\begin{eqnarray*}
\|\Do(h_a)\|_{L^\infty(\Om)}  && \geq \sup_{x \in \Om}\left( |h_a(x)| \,  \fluence(x)  -  \mua(x) \ell_\infty(x) \fluence(x) \snorm{h_a}_{L^\infty(\Om_0)}  \right) \\ && \geq \sup_{x \in\Om} \big(|h_a(x)|  -  \mua (x) \ell_\infty (x)  \, \snorm{h_a}_{L^\infty(\Om_0)} \big) \, \fluence(x).
\end{eqnarray*}
Assume that $h_a \in L^2(\Om_0)$ is not identically zero. Since $\|\mua \,\ell_\infty \|_{L^\infty(\Om_0)}<1$,  we can find $x \in \Om_0$ such that  $|h_a(x)|  -  \mua (x) \ell_\infty(x)  \snorm{h_a}_{L^\infty(\Om_0)} >0$.
We arrive at $\|\Do(h_a)\|_{L^\infty(\Om)} >0$.
Therefore, $\Do$ is injective.

\ref{thm:sigma0-stablity1-no}, \ref{thm:sigma0-stablity2-no}
Analogously to Theorem~\ref{thm:sigma0-stablity}.
\end{proof}

\subsection{Non-vanishing scattering}
\label{sec:non}

In this section, we consider the case of known but non-vanishing
scattering $\mus \neq 0$. Let us consider the case of  single illumination.
We  present a stability and uniqueness result for the
linearized heating operator
\begin{equation}  \label{eq:derH0}
\Do (h_a) \coloneqq \Ho' \skl{\mu}(h_a, 0)
= \phi \,  h_a -  \mua\int_{\sph^{d-1}}  \Psi(\edot,\theta) \, \rmd  \theta
\end{equation}
 as well as for the linearized forward operator $\Wo\Do$.  Here $\mu = (\mua, \mus) \in \dom$
 is the linearization point, $\Phi \in W^\infty(\Om\times \sph^{d-1})$ the solution of  \eqref{eq:strong}, and $\fluence \coloneqq \int_{\sph^{d-1}} \Phi(\edot,\theta) \, \rmd \theta$ the background fluence.
  Further, $\Psi \in W^\infty_0(\Om\times \sph^{d-1})$ satisfies
 $(\Vo_0 - \mus \Ko) \Psi = h_a \, \Phi$. The latter  equation can equivalently be rewritten in the form
\begin{equation} \label{E:new-form}
(\Io- \Vo_0^{-1} \mus \Ko) \Psi =  \Vo_0^{-1} h_a \Phi \,.
\end{equation}
Recall that $\fluence_{\min}$ defined  by \eqref{eq:phimin} is the minimum
of the background fluence, and   $\ell_+(x)$ defined by \eqref{eq:ell}  is the average of
 $\ell(x,\theta)$ over all directions $\theta \in \sph^1$.

\begin{lemma}\label{lem:scatt} Assume that $\Phi  \in C^\infty(\overline \Om \times \sph^{d-1})$.
\begin{enumerate}
\item\label{lem:scatt1} $\Do$ is a pseudo-differential operator of order $0$ with
principal symbol $\phi$.
\item\label{lem:scatt2} $\Do$ is injective on $L^\infty(\Om_0)$, if
\begin{equation} \label{E:Ineq2}   \sabs{\sph^{d-1}} \, \snorm{\mua}_\infty \,\diam(\Om) \,  e^{\snorm{\mus \ell_+}_\infty} \,   \snorm{\Phi}_\infty <\fluence_{\min} \,.
\end{equation}
\end{enumerate}
\end{lemma}

\begin{proof}
\ref{lem:scatt1}
Repeating arguments of \cite{EggSch14b}, we conclude
$\|\Vo_0^{-1} \,\mus \Ko\| \leq 1- e^{- \snorm{\mus \ell_+}_\infty}$.
Therefore, the operator $\Io- \Vo_0^{-1} \mus \Ko$ is invertible with
\begin{equation} \label{eq:Vseries}
(\Io- \Vo_0^{-1} \mus \Ko)^{-1}
= \sum_{k=0}^\infty (\Vo_0^{-1} \mus \Ko)^k \,,
\end{equation}
and thus $\|(\Io- \Vo_0^{-1} \mus \Ko)^{-1} \|
\leq  e^{\|\mus \ell_+ \|_\infty}$.
From (\ref{E:new-form}) and \eqref{eq:Vseries} we obtain
the equality $\Psi = \sum_{k=0}^\infty (\Vo_0^{-1} \mus \Ko)^k \, \Vo_0^{-1} (h_a \Phi)$.
Repeating the argument in the proof of Theorem~\ref{lem:psi}, we obtain that
$h_a \mapsto \int_{\sph^{d-1}} \Psi(x,\theta) \rmd \theta$ is a pseudo-differential
operator of order at most $-1/2$ which  yields the assertion.

\ref{lem:scatt2}
From (\ref{eq:derH0}) we see that in order to prove the uniqueness of $\Do$ it
suffices to show  $\|\phi \, h_a  \|_\infty > \|\mua\int_{\sph^{d-1}}  \Psi(\edot,\theta) \, \rmd  \theta\|_\infty$ for $h_a \neq  0$. The left hand side is bounded from below by
$\snorm{h_a}_\infty \, \fluence_{\min} $, while
the right hand side is bounded from above by
\begin{equation*} \snorm{\mua (\Io- \Vo_0^{-1} \mus \Ko)^{-1}  \Vo_0^{-1} (h_a \Phi)}_\infty
 \leq \snorm{\mua}_\infty  \|(\Io- \Vo_0^{-1} \mus \Ko)^{-1} \| \snorm{\Vo_0^{-1}}_{L^\infty, L^\infty} \|h_a \Phi\big\|_\infty.\end{equation*}
Recalling that
 $ \snorm{\Vo_0^{-1}}_{L^\infty, L^\infty} \leq  \diam(\Om) $
 and   $ \|(\Io- \Vo_0^{-1} \mus \Ko)^{-1} \| \leq  e^{\|\mus \ell_+ \|_\infty}$
 (see (\ref{E:norm}) and the line below (\ref{eq:Vseries})) and making use of (\ref{E:Ineq2}), we obtain
\begin{eqnarray*} \sabs{\sph^{d-1}} \, \snorm{\mua (\Io- \Vo_0^{-1} \mus \Ko)^{-1}  \Vo_0^{-1} (h_a \Phi)}_\infty < \snorm{h_a}_\infty \, \fluence_{\min} \,. \end{eqnarray*}
This finishes our proof.
\end{proof}

Similar to the  case of  vanishing scattering we
obtain the following results.

\begin{theorem} \label{thm:scatt}  Suppose $\phi(x)>0$ for all $x \in \Om_0$.
\begin{enumerate}
\item\label{thm:scatt1} $\wf(h) = \wf(\Do) \cap \cT^* \Om_0$.
\item\label{thm:scatt2} $\Do \colon  L^2(\Om_0) \to L^2(\Om)$ is a Fredholm operator.
\item\label{thm:scatt3} $\dim \kl{\ker \Do} < \infty$.

\item\label{thm:scatt4}  If, additionally, inequality~\eqref{E:Ineq2} holds, then there exist constants
$C_1, C_2>0$ such that for  all $h_a \in L^2(\Om_0)$ we have
\begin{align}
\frac{1}{C_1}\snorm{h_a}_{L^2(\Om_0)} &\leq
\|\Do h_a\|_{L^2(\Om)} \leq C_1 \snorm{h_a}_{L^2(\Om_0)}  \,,\\
\frac{1}{C_2}\snorm{h_a}_{L^2(\Om_0)} &\leq
\| \Wo \Do h_a\|_{L^2(\Om)} \leq C_2 \snorm{h_a}_{L^2(\Om_0)} \,.
\end{align}
\end{enumerate}
\end{theorem}

\begin{proof}
\ref{thm:scatt1}-\ref{thm:scatt3}: According to Lemma~\ref{lem:scatt} \ref{lem:scatt1}, $\Do$ is a
elliptic pseudodifferential operator of order zero with principal symbol $\phi > 0$,
which implies \ref{thm:scatt1}-\ref{thm:scatt3}.

\ref{thm:scatt4}  Is shown analogously to Theorem \ref{thm:sigma0-stablity}.
\end{proof}

\subsection{Multiple illuminations}

We now consider the general case of possibly multiple illuminations
$\qo_i$ for $i = 1, \dots N$, where we assume $\mus=0$ and $\qi =0$.
The observation surface  and measurement times for the $i$-th illumination
are denoted by $\Lambda_i$ and $T_i$; see Subsection~\ref{sec:forward-multiple}.

For any $i=1,\dots, N$ let us denote
\begin{align*}
\Om_i  &\coloneqq  \{ x \in \Om \mid  \dist(x,\Lambda_i) \leq T_i \text{ and } \phi_i(x) > 0 \} \,,\\
\Sigma_i   &\coloneqq \{ (x,\xi) \in \Om \times (\R^d \setminus 0) \mid
\text{ the line passing through $x$  along  direction $\xi$}   \\
& \hspace{0.13\textwidth} \text{   intersects $\Lambda_i$ at a distance less than $T_i$ from $x$ } \text{ and } \phi_i(x)>0 \} \,.
\end{align*}
Then $\Om_i$ is the uniqueness set  and $\Sigma_i$ the visibility set
determined by the  $i$-th illumination  with observation surface $\Lambda_i$ and
measurement time $T_i$.   We also denote $\Sigma \coloneqq  (\Sigma_1, \dots, \Sigma_N)$ and
set $\|\Wo \Do(h_a)\|_{L^2(\Sigma)}^2
    \coloneqq \sum_{i=1}^N \|\Wo_i \Do_i(h_a)\|^2_{L^2(\Sigma_i)}$.

\begin{theorem} \label{thm:multiple}
Suppose   $\|\mua \, \ell_\infty \|_{L^\infty(\Om_0)} <1$.
\begin{enumerate}
\item\label{thm:multiple1} If $\Om \subseteq \bigcup_i \Om_i$, then $\Wo \Do$ is injective.

\item\label{thm:multiple2} If $\Om \times (\R^d \setminus 0)  \subseteq \bigcup_i \Sigma_i$, then
there is $C>0$ such that
\begin{equation*}
\forall h_a \in L^2(\Om_0) \colon \quad
\frac{1}{C}\snorm{h_a}_{L^2(\Om_0)} \leq \|\Wo \Do(h_a)\|_{L^2(\Sigma)} \leq C \snorm{h_a}_{L^2(\Om_0)}
\end{equation*}
\end{enumerate}
\end{theorem}

\begin{proof}
\ref{thm:multiple1} Suppose $h_a \in L^2(\Om_0)$ does not completely vanish.
Repeating the argument in the proof  of Theorem~\ref{thm:no}, we obtain
\begin{eqnarray*}
\|\Do_i(h_a)\|_{L^\infty(\Om_i)}  \geq \sup \set{ \big(|h_a(x)|  -  \mua (x) \ell_\infty (x)  \, \snorm{h_a}_{L^\infty(\Om_0)} \big) \, \fluence_i(x)
\mid x \in\Om_i} \,.
\end{eqnarray*}
Since it holds that $\|\mua \,\ell_\infty \|_{L^\infty(\Om_0)}<1$,  we can find some
$x \in \Om_0$ such that  $|h_a(x)|  -  \mua (x) \ell_\infty(x)
\snorm{h_a}_{L^\infty(\Om_0)} >0$.
As $\Om \subseteq \bigcup_i \Om_i$, we have $x \in \Om_i$ for some $i$. We arrive at $\|\Do_i(h_a)\|_{L^\infty(\Om_i)} >0$. Since $\Wo_i$ is injective on $\Om_i$, we obtain $\Wo_i \Do_i (h_a) \neq 0$.
Therefore, $\Wo \Do$ is injective.

\ref{thm:multiple2} Repeating the argument for Theorem~\ref{thm:no}, we obtain
$\|\Wo_i \Do_i(h_a)\|_{L^2(\Sigma_i)} \leq C \snorm{h_a}_{L^2(\Om)}$,
 and therefore $\|\Wo \Do(h_a)\|_{L^2(\Sigma)} \leq C \snorm{h_a}_{L^2(\Om_0)}$
 for some $C>0$.
It now remains to prove the left in equality in (b). To this end, let us notice that
$$\|h_a\|_{L^2(\Om_i)} \leq C \, (\|\Do_i h_a\|_{L^2(\Om_i)} + \, \|h_a\|_{H^{-1/2}(\Om_0)})
\,. $$
Due to the stability of the wave inversion, we have
$ \|\Do_i h_a\|_{L^2(\Om_i)} \leq C_i \|\Wo_i \Do_i h_a\|_{L^2(\Sigma_i)}$
Therefore $\|h_a\|_{L^2(\Om_i)} \leq C_i \, (\|\Wo_i \Do_i h_a\|_{L^2(\Sigma_i)} + \, \|h_a\|_{H^{-1/2}(\Om_0)})$ which  gives
$$\|h_a\|_{L^2(\Om_0)} \leq C
\, (\|\Wo \Do h_a\|_{L^2(\Sigma)} + \, \|h_a\|_{H^{-1/2}(\Om_0)}) \,.$$
Since the map $\Wo \Do$ is injective, applying \cite[Proposition V.3.1]{taylor1981pseudodifferential}, we conclude the estimate
$\|h_a\|_{L^2(\Om_0)} \leq C \, \|\Wo \Do h_a\|_{L^2(\Sigma)}$.
\end{proof}

\begin{remark}[Unknown scattering]
Suppose that the attenuation and  scattering are unknown and
consider the linearization with respect to both parameters
\begin{equation*}
\Ho' \skl{\mua,\mus}(h_a, h_s)
= \phi  \, h_a - \mua\int_{\sph^{d-1}}  \Psi(\edot,\theta) \, \rmd  \theta
\,,
\end{equation*}
where $\Psi =  \kl{ \Vo_0 - \mus \Ko}^{-1} \bigl[ \skl{h_a + h_s - h_s\Ko} \Phi  \bigr]$.
The second term in the displayed expression
is a smoothing operator of degree  at least $1/2$. Let $h_a$ and $h_s$ have
singularities of the same order (say, they both have jump singularities). Then, the main singularities of  $\Ho' \skl{\mua,\mus}$ come from the term $\phi \, h_a$.
In the case that $\phi>0$ on $\overline \Om$, then all the singularity of $h_a$ are reconstructed with the correct order and magnitude by $\tfrac{1}{\phi} \Ho' \skl{\mua,\mus}(h_a, h_s)$. A similar situation occurs in the case of multiple illuminations. This indicates that recovering  the scattering coefficient is more ill-posed than recovering the attenuation coefficient.
\end{remark}

\section{Numerical simulations}
\label{sec:num}

Simulations are performed in spatial dimension $d=2$.
The linearized RTE is solved on a  square domain $\Om={ [-1,1]^2}$.
For the scattering  kernel we choose the two dimensional version of the Henyey-Greenstein kernel,
\begin{equation*}
	k(\theta,\theta')
	\coloneqq
	\frac{1}{2\pi}\frac{1-g^2}{1+g^2-2g\cos(\ip{\theta}{\theta'})}
	\quad \text{ for } \theta, \theta' \in \sph^1 \,,
\end{equation*}
where $g \in (0,1)$ is the  anisotropy factor.
For all simulations we choose the internal sources $q$ to be zero.
Before we present results of our numerical simulations we first outline how we numerically solve the stationary RTE in two spatial dimensions.
This step is required for simulating the data as well as for evaluating the adjoint  of the linearized problem in the iterative solution.

\subsection{Numerical solution of the RTE}
\label{sec:FE}

For solving the linearized RTE for the inverse problem \eqref{eq:der} we employ a streamline diffusion finite element method as in \cite{haltmeier2015single,YaoSunHua10}. The weak form of equation \eqref{eq:strong}
is derived by integration against a test function $w \colon \Om\times \sph^1 \to \R$.
Integrating by parts in the transport term yields
\begin{multline}\label{eq:weak}
\int_{\Om}
\int_{\sph^1}
\kl{-\ip{\theta}{\nabla_x} w + \mua w+ \mus w- \mus  \Ko w}
  \Phi  \, \rmd \theta \, \rmd x
\\ +\int_{\partial \Om \times \sph^1} \Phi \, w \, \kl{\ip{\theta}{\nu}}
\rmd\sigma
=
\int_{\Om}
\int_{\sph^1}\!\!
q \, w  \, \rmd \theta \, \rmd x  \,,
\end{multline}
where  we dropped all dependencies on the variables  and $\rmd \sigma$ denotes the usual surface measure on  $\partial \Om \times \sph^1$.
Our numerical  scheme replaces the  exact solution $\Phi$ by a linear combination
$\Phi^{(h)} (x,\theta)
	=
	\sum_{i =1}^{N_h} c_i^{(h)}  \psi_i^{(h)} (x,\theta)$ in the finite element space, where any basis function $\psi_i^{(h)}(x,\theta)$ is the product of a basis function in the
spatial variable $x$ and a basis function in the angular variable $\theta$. We use a uniform triangular grid of grid size $h$, that leads to basis functions that are pyramids; see~\cite[Figure 3]{haltmeier2015single}.
To discretize the velocity direction we divide the  unit circle into $N_\theta$ equal sub-intervals and choose the basis functions to be piecewise affine and continuous functions.

The streamline diffusion method \cite{Kan10} adds some artificial diffusion in the transport direction to increase stability in low scattering areas. It uses the test functions
$	w(x,\theta) = \sum_{j =1}^{N_h} w_j\skl{\psi_j(x,\theta) +
	D(x,\theta)\, \ip{\theta}{\nabla_x}\psi_j(x,\theta)}$,
	where $D(x,\theta)$ is an appropriate stabilization parameter.
In our experiments we choose the stabilization parameter
$D(x,\theta) = 3 h/100$ for areas where $\mua(x)+\mus(x)  < 1$ and zero otherwise. Using the  test functions  in equation \eqref{eq:weak} one obtains a system of  linear equations $M^{(h)} c^{(h)} = b^{(h)}$ for the coefficient vector of the numerical solution.

\subsection{Test scenario for single illumination}
\label{sec:test}

We illuminate the sample in the orthogonal direction along the lower boundary of the rectangular
domain, we  choose
\begin{equation} \label{eq:oil}
f(x,\theta)  = I_0  F (\theta ) \times
\begin{cases}
	1 &  \text{ for } (x_1,x_2) \in \{-1\}\times[-1,1] \,, \\
	0 &  \text{ on the rest of } \partial \Omega   \,,
\end{cases}
\end{equation}
where $F$ is constant  on each of the $N_\theta$ sub-intervals of unit circle,
takes the value   $N_\theta/(2\pi)$ at $(0,1)^\trans$ and is zero at the other discretization points.
In the case of multiple illuminations we use orthogonal illuminations from all four sides of
$ \Omega = [-1,1]^2$, where the absorption and  scattering  coefficients are supported. The forward problem for the wave equation is solved by discretizing the  integral representation~\eqref{eq:waveoperator}.
We take measurements  of the pressure on the circle of radius $R = 1.5$ and the  pressure data on the time interval $[0,3]$.

In Figure~\ref{fig:num} we illustrate the measurement procedure for full data, where $\Lambda = \partial B_{1.5}(0)$. For partial measurements we restrict the polar angle $\varphi$ on the measurement circle to $[0,\pi]$. Thus illumination and measurement are performed on the same side of the sample,
a situation allowing for obstructions on the  side opposite to the performed measurements.
We add $0.5\%$ random noise to the simulated data; more precisely we take the maximal value of the simulated pressure and add white noise with a standard deviation of $0.5\%$
of that maximal value.

\begin{figure}[h!] \centering
\begin{subfigure}[b]{0.4\linewidth}
\begin{tikzpicture}[scale=0.8]
\begin{axis}[x=1cm,y=1cm, axis equal, clip=false, domain=-2.5:2.5,
xtick=\empty,
ytick=\empty, %title=Measurement scenario
]
\addplot graphics [xmin=-1.35,xmax=1.35,ymin=-1.35,ymax=1.35, includegraphics={trim=7cm 2.9cm 6cm 2cm,clip}] {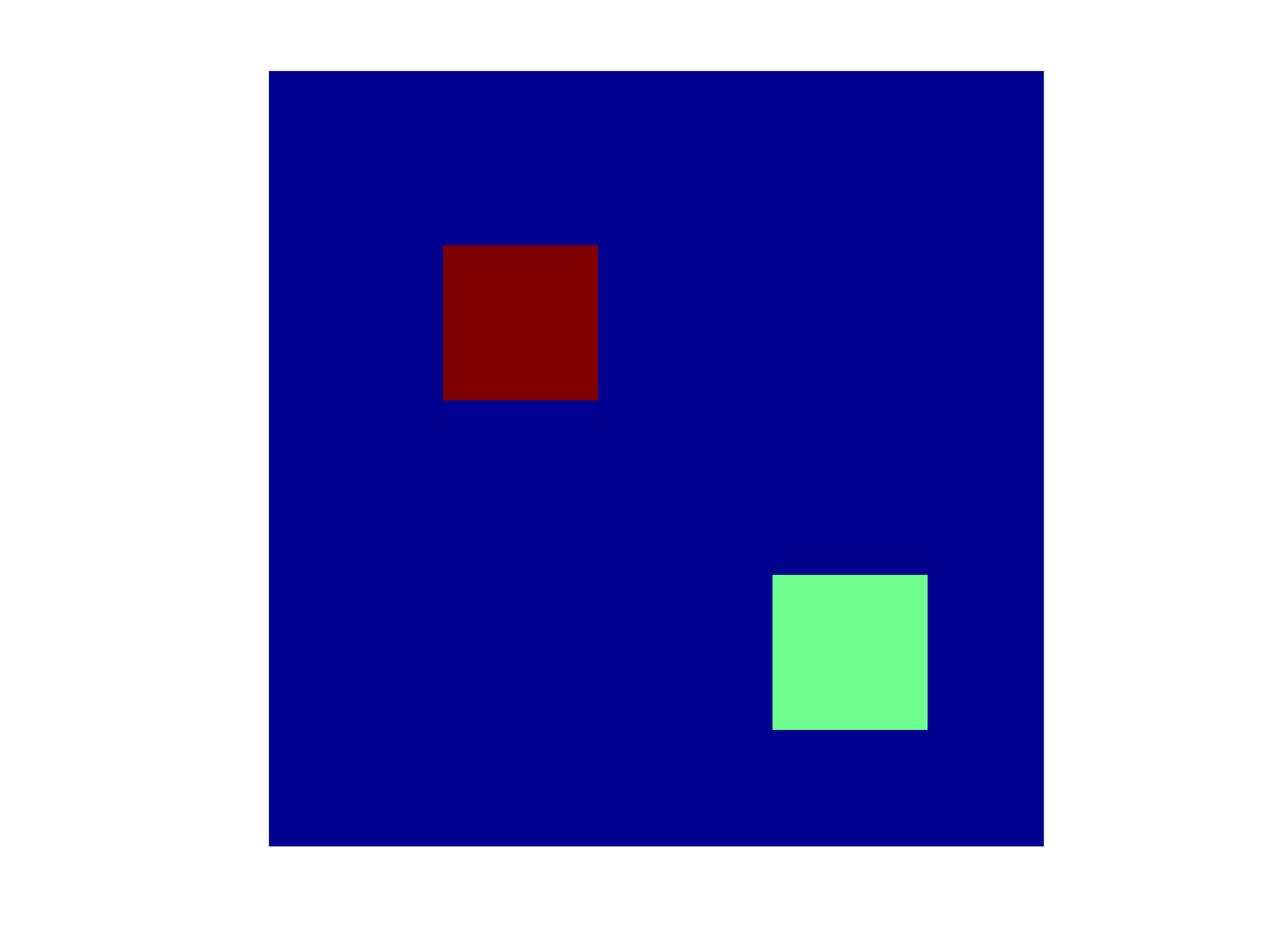};
\addplot[samples=20, domain=-1:1,
% the default choice 'variable=\x' leads to unexpected results here!
variable=\t,
quiver={
u={0},
v={0.5},
scale arrows=0.5},
->,blue]
({1.5*t}, {-1.75});
\addplot[samples=100, domain=0:pi, dashed]
({2.5*cos(deg(x))}, {2.5*sin(deg(x))});
%\node[pin=90:$\Theta$] at ({sqrt(2)+0.25},{sqrt(2)-.25}) {}; %Benennung Winkel
\addplot[samples=100, domain=pi:2*pi]
({2.5*cos(deg(x))}, {2.5*sin(deg(x))});
\addplot[samples=20, domain =0:1] ({2.5*cos(315)+x},{2.5*sin(315)});
\addplot[samples=20, domain =0:1] ({(2.5+x)*cos(315)},{(2.5+x)*sin(315)});
\addplot[samples=20, domain =0:45, <-] ({2.5*cos(315)+0.8*cos(315+x)},{2.5*sin(315)+0.8*sin(315+x)});
\node at (axis cs:1.25*1.8 , -1.25*1.6) {$\varphi$};
\node at (axis cs:0,2.7) {Detectors $\Lambda$};
\node at (axis cs:0,-1.95) {\textcolor{blue}{Illumination}};
\end{axis}
\end{tikzpicture}
\caption{Measurement setup.}\label{fig:setup}
\end{subfigure}
\begin{subfigure}[b]{0.4\linewidth}
\includegraphics*[width=0.95\textwidth, trim=0.5cm 1cm 2.9cm 1.5cm,clip]{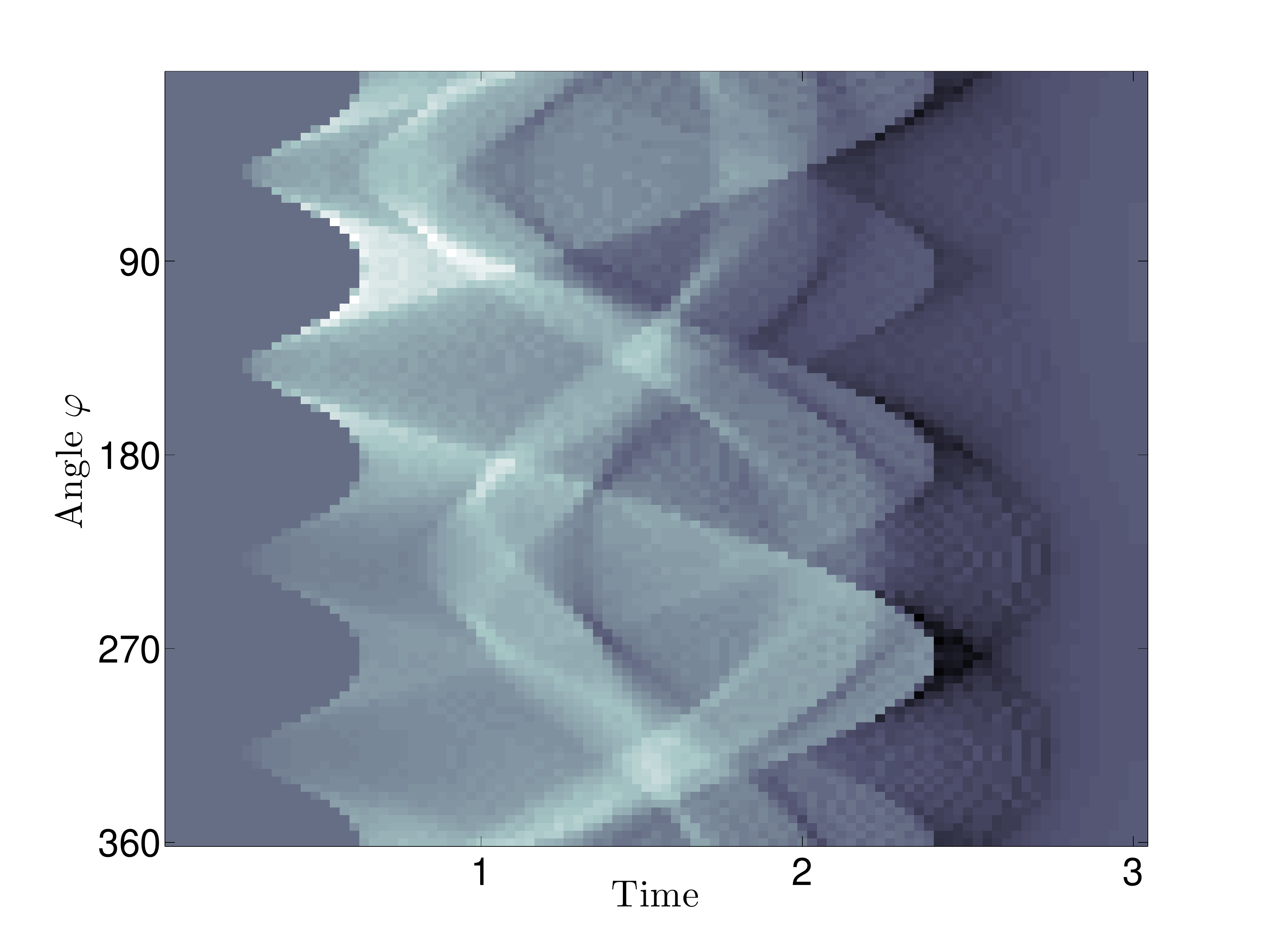}
%\end{minipage}
\caption{Simulated pressure data.}\label{fig:measurement}
\end{subfigure}
\caption{\textsc{Measurement\label{fig:num} setup and simulated data.} Acoustic pressure is represented by the density in the gray scale.}
\end{figure}
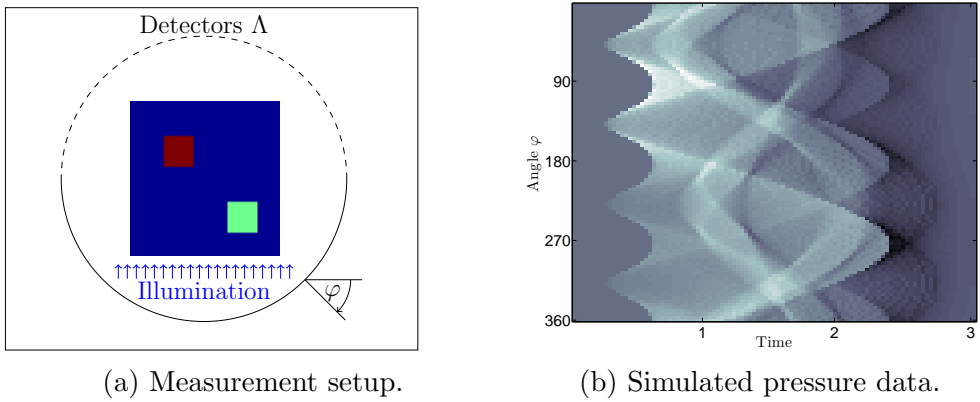

\subsection{Solution of the linearized inverse problem}

We solve the forward problem and added random noise as described  in the previous subsection.
Since we use only boundary sources of illumination in our simulation this corresponds to
calculating the simulated forward data
\[
\data \coloneqq \Wo_{\La,T}\circ \Ho_{f,0}(\mua,\mus) + \mathrm{noise}\,.
\]
In our simulations the parameters $\mua$ and $\mus$ are constant on the boundary. The linearization point $(\mua^\lin,\mus^\lin)$ is chosen spatially constant and equal to the respective parameter on the boundary.
This represents a situation where the parameters of the tissue on the boundary are known but internal variations in absorption and/or scattering are of interest. The solution at the linearization point is denoted by
$v^\lin\coloneqq\Wo_{\La,T}\circ \Ho_{f,0}(\mua^\lin,\mus^\lin)$.
In our numerical simulations we do not attempt to reconstruct $\mus$. We linearize around the value at the boundary and fix that value for all our calculations. Reconstruction of $\mus$ is more difficult because $\mus$ influences the heating in a more indirect way. Preliminary numerical attempts indicate that different step sizes in the direction of $\mua$ and $\mus$ have to be used but reconstruction of also $\mus$ is a subject of further studies.

As in Subsection~\ref{sec:non} we write
$\Do(h_a)=(\Ho_{f,0})'(\mua^\lin,\mus^\lin)(h_a,0)$
for the  G\^ataux derivative  of $\Ho_{f,0}$ at position
$(\mua^\lin,\mus^\lin)$ in direction $(h_a,0)$; see
Proposition~\ref{prop:Hdiff}.  We refer to $\Do$ as the linearized heating operator.
The solution of the linearized inverse problem then  consists in finding $h_a$
in the set of admissible directions, such that the residuum functional of the linearized problem
is minimized,
\begin{equation}\label{eq:linMin}
\mathrm{Res}(h_a)=\|\data-v^\lin - ( \Wo_{\La,T}\circ \Do) (h_a)\|^2_2\rightarrow \min  \,.
\end{equation}
Using the minimizer of the linearized residuum we  define  the  approximate linearized solution  as
$\mua  = \mua^\lin + h_a$.

For solving \eqref{eq:linMin} we use  the Landweber iteration
\begin{equation} \label{eq:landweber}
h^{n+1}_a = h_a^n + \lambda \,  \Do^\ast \circ \Wo_{\La,T}^\ast\left(\data-v^\lin-
(\Wo_{\La,T}\circ \Do) (h_a^n)\right) \quad \text{ for } n \in \N\,,
\end{equation}
with starting value $h_a^0=0$, where $\lambda$ is  the step size.
Note that we do not reconstruct the heating as intermediate step but immediately reconstruct
the parameters of interest. Such a single stage approach has advantages especially in the case of multiple measurements with partial data; a more thorough discussion can be found in \cite{haltmeier2015single} for example.
The heating operator is discretized with the same finite element technique as the forward-problem
(see Proposition~\ref{prop:Hdiff}) and the adjoint is calculated after discretization as an adjoint matrix.
To solve the adjoint wave propagation problem we discretize formula \eqref{eq:adjoint}.

Our stability analysis shows that $\Wo_{\La,T}\circ \Do$  is Fredholm operator
 (see Theorem~\ref{thm:scatt}) and, in particular, that $\Wo_{\La,T}\circ \Do$ has closed range.
 Therefore,  for any $\data \in L^2(\La \times (0,T))$, the Landweber iteration converges to
 the minimizer of~\eqref{eq:linMin} with a linear rate of convergence, provided that the step
 size satisfies $\la < 2 / \|\Wo_{\La,T}\circ \Do\|_2^2$.
 In the numerical simulations we used  about $50$ iterations after which we
 already obtained quite accurate results. The convergence speed can further be
 accelerated by using iterations  such as the CG algorithm. See
 \cite{haltmeier2016iterative} for a comparison and analysis of various  iterative
 methods for the  wave inversion  process.
Theorem~\ref{thm:scatt} further implies that the minimizer of~\eqref{eq:linMin} is unique
and satisfies the  two-sided stability estimates if~\eqref{E:Ineq2}
is satisfied. Evaluated at the linearization point $\mu^\lin$
inequality~\eqref{E:Ineq2} reads
$
2 \pi \, \snorm{\mua^\lin}_\infty \, e^{\snorm{\mus^\lin \ell_+}_\infty} \,  \diam(\Om) \, \snorm{\Phi^\lin}_\infty < \min_{x\in\Omega} \left\{ \int_{\sph^1}\Phi^\lin(x,\theta)\rmd \theta \right\}
$.
Assuming, as is reasonable for collinear illumination, that $\Phi^\lin$ takes its maximum at the boundary we find $2\pi \, \diam(\Om) \, \snorm{\Phi^\lin}_\infty = I_0 N_\theta 2 \sqrt{2}$.
Here we have taken  $f$ as in \eqref{eq:oil} which satisfies  $\snorm{f}_\infty = I_0 N_\theta/(2\pi)$.
Geometric considerations show that maximum of $\ell_+$ is taken at the center where
it takes the  value $4 \operatorname{arsinh}(1) /\pi $.
So, for constant $\mus^\lin$ and $N_\theta=64$, the above inequality  simplifies to
\begin{equation} \label{eq:uc}
\snorm{\mua^\lin}_\infty \cdot 3^{\mus^\lin}
\lesssim
\frac{1}{181 \, I_0 } \min_{x\in\Omega} \left\{ \int_{\sph^1}\Phi^\lin(x,\theta)\rmd \theta \right\} \,.
\end{equation}
Condition \eqref{eq:uc} requires quite small values for absorption and scattering at the linearization point and future work will be done
to weaken  this condition.

\subsection{Numerical  results}

The domain $\Om =  [-1,1]^2$ is discretized  by a mesh of triangular elements of 6400 degrees of freedom and we divide the angular domain into $N_\theta = 64$ sub-intervals of equal length. The anisotropy factor is taken as $g = 0.8$ throughout all the experiments.

\begin{figure}[thb!]\centering
\begin{subfigure}[b]{0.3\linewidth}
\includegraphics*[width=\textwidth, trim=4cm 3cm 0cm 1cm,clip]{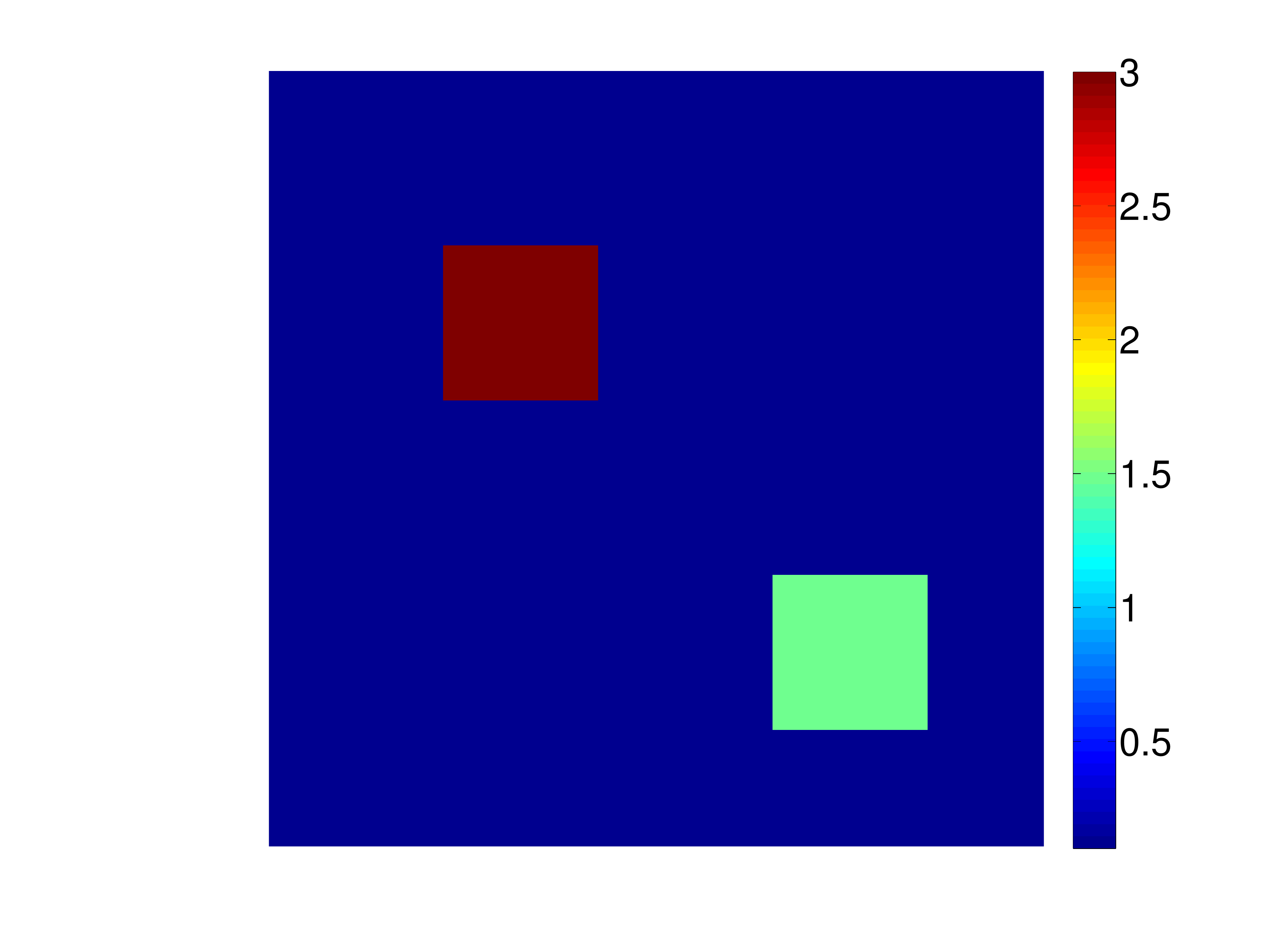}
\caption{Phantom $\mua$}\label{fig:e11}
\end{subfigure}
\begin{subfigure}[b]{0.3\linewidth}
\includegraphics*[width=\textwidth, trim=4cm 3cm 0cm 1cm,clip]{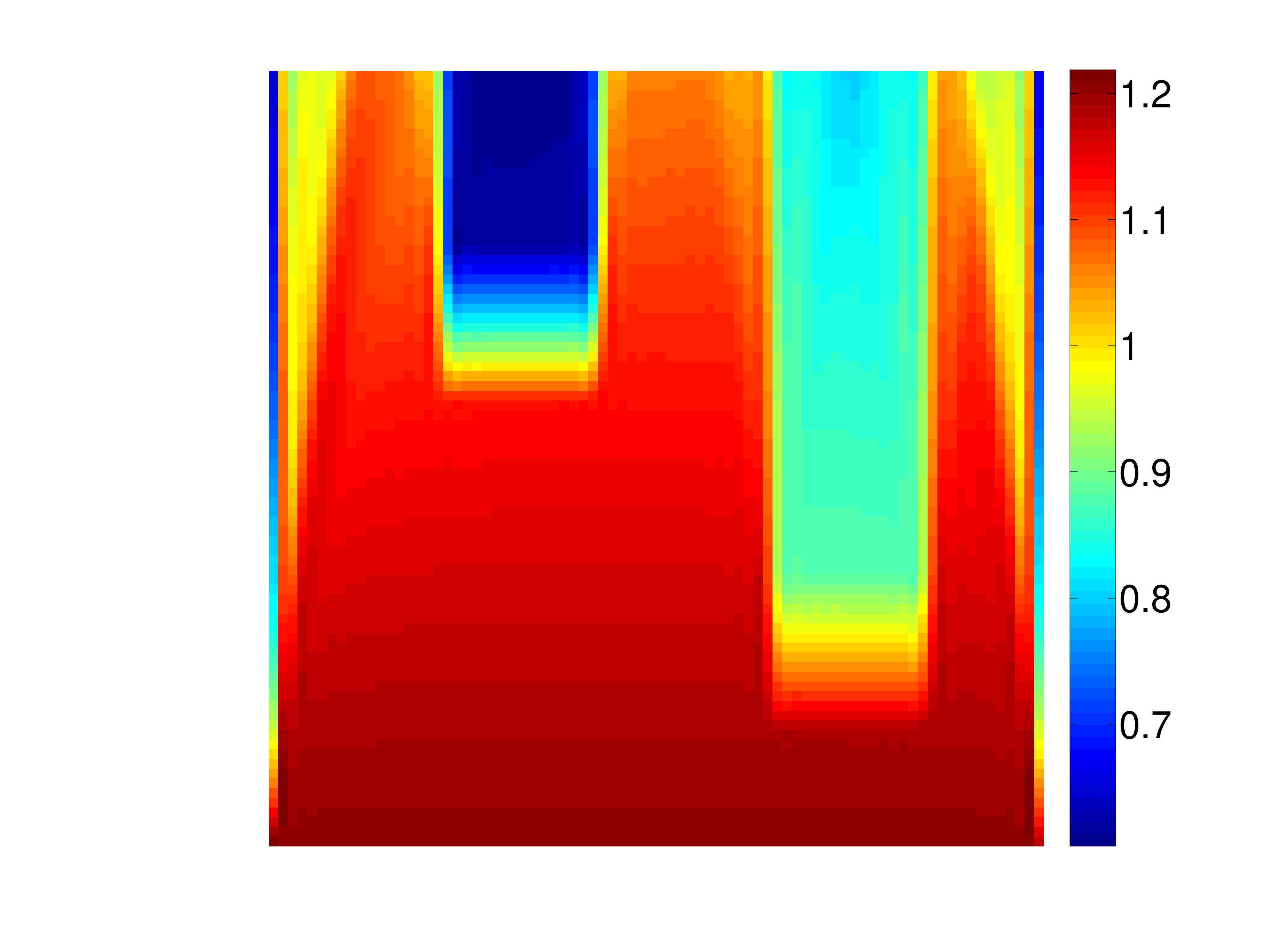}
\caption{Illumination}\label{fig:e12}
\end{subfigure}
\begin{subfigure}[b]{0.3\linewidth}
\includegraphics*[width=\textwidth, trim=17cm 2.5cm 2cm 14cm,clip]{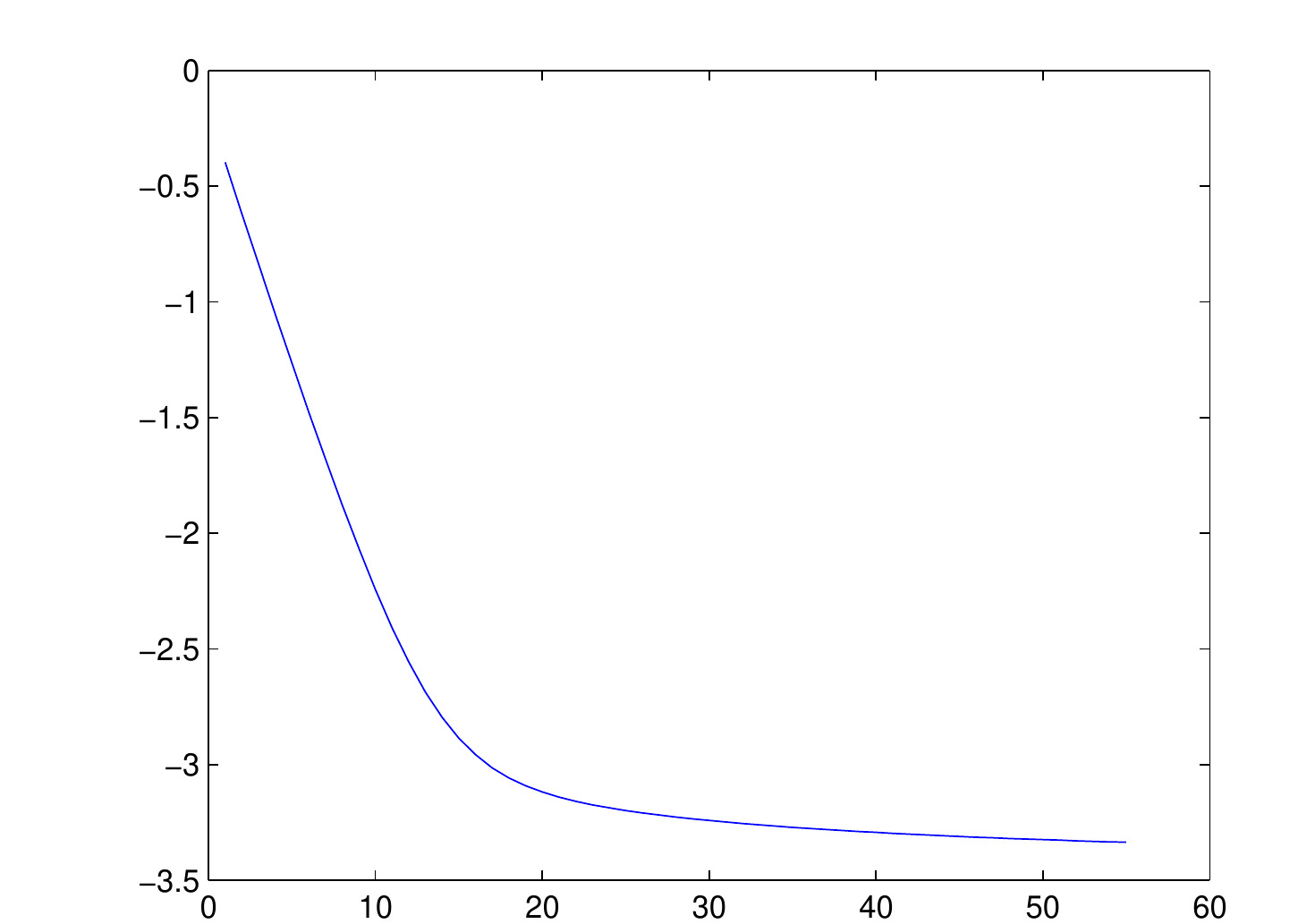}
\caption{Residual.}\label{fig:e13}
\end{subfigure}\\
\begin{subfigure}[b]{0.3\linewidth}
\includegraphics*[width=\textwidth, trim=4cm 3cm 0cm 1cm,clip]{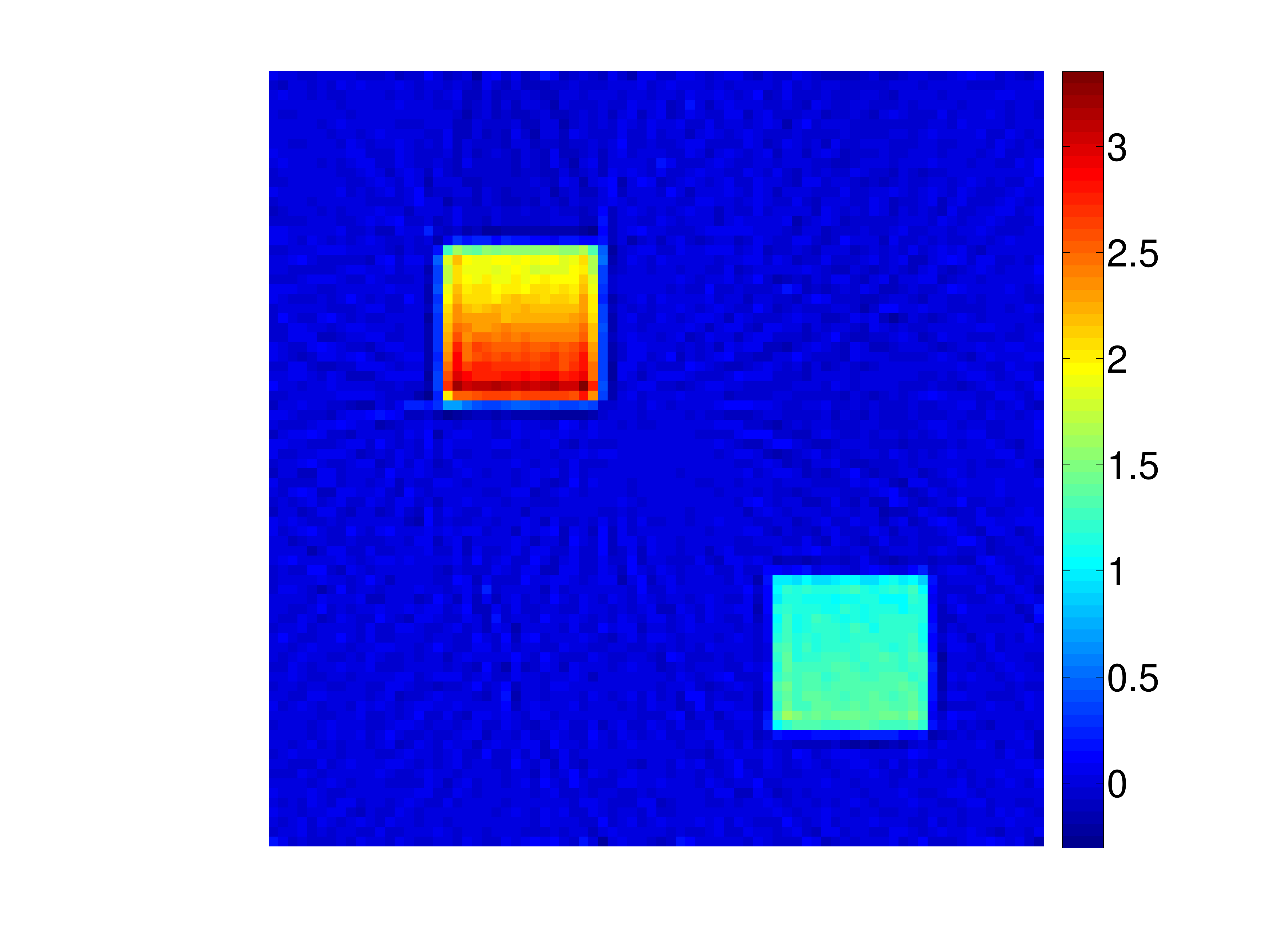}
\caption{Full data}\label{fig:e14}
\end{subfigure}
\begin{subfigure}[b]{0.3\linewidth}
\includegraphics*[width=\textwidth, trim=4cm 2.5cm 0cm 1cm,clip]{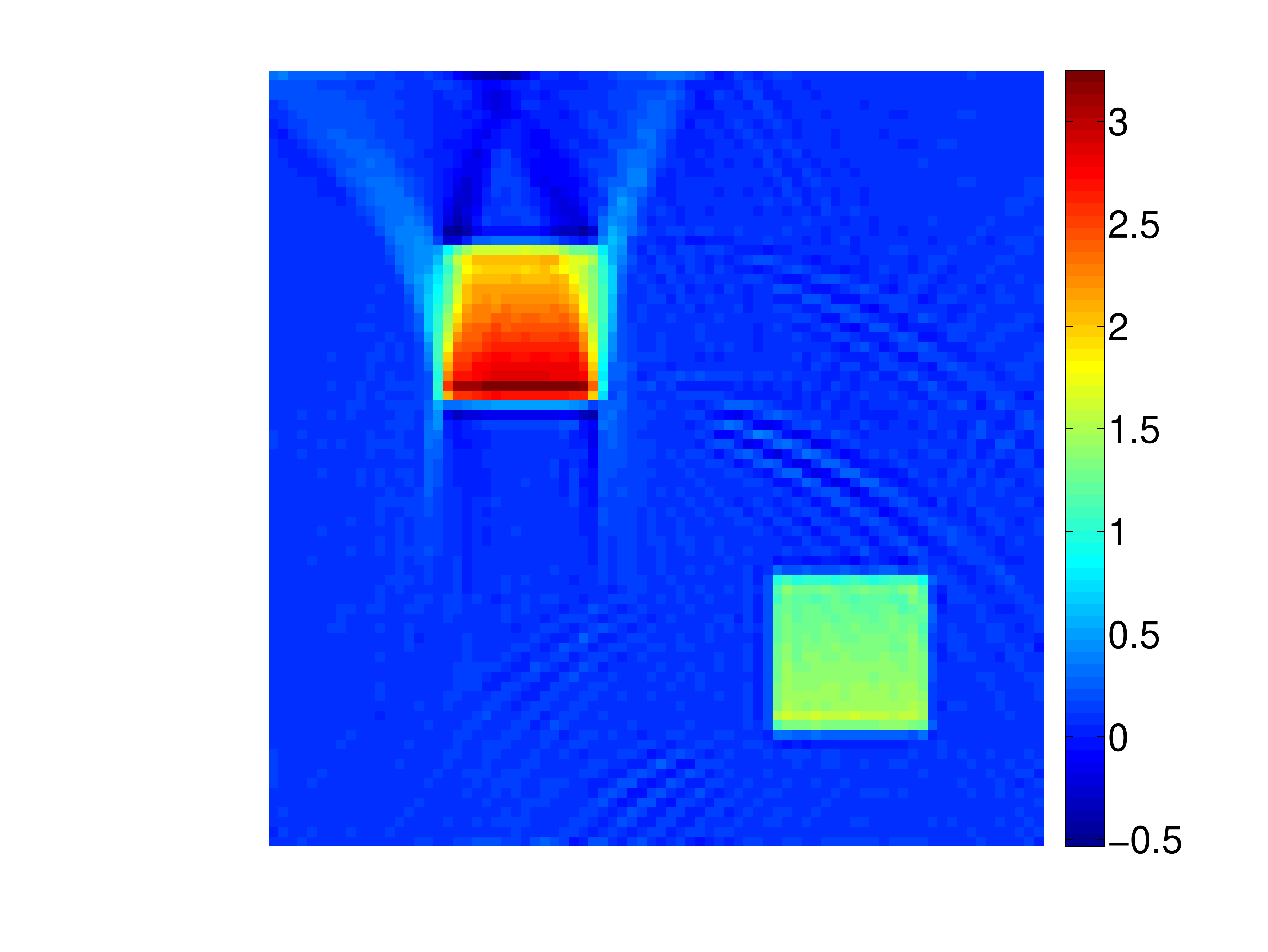}
\caption{Half data. \label{fig:e15}}
\end{subfigure}
\begin{subfigure}[b]{0.3\linewidth}
\includegraphics*[width=\textwidth, trim=4cm 2.5cm 0cm 1cm,clip]{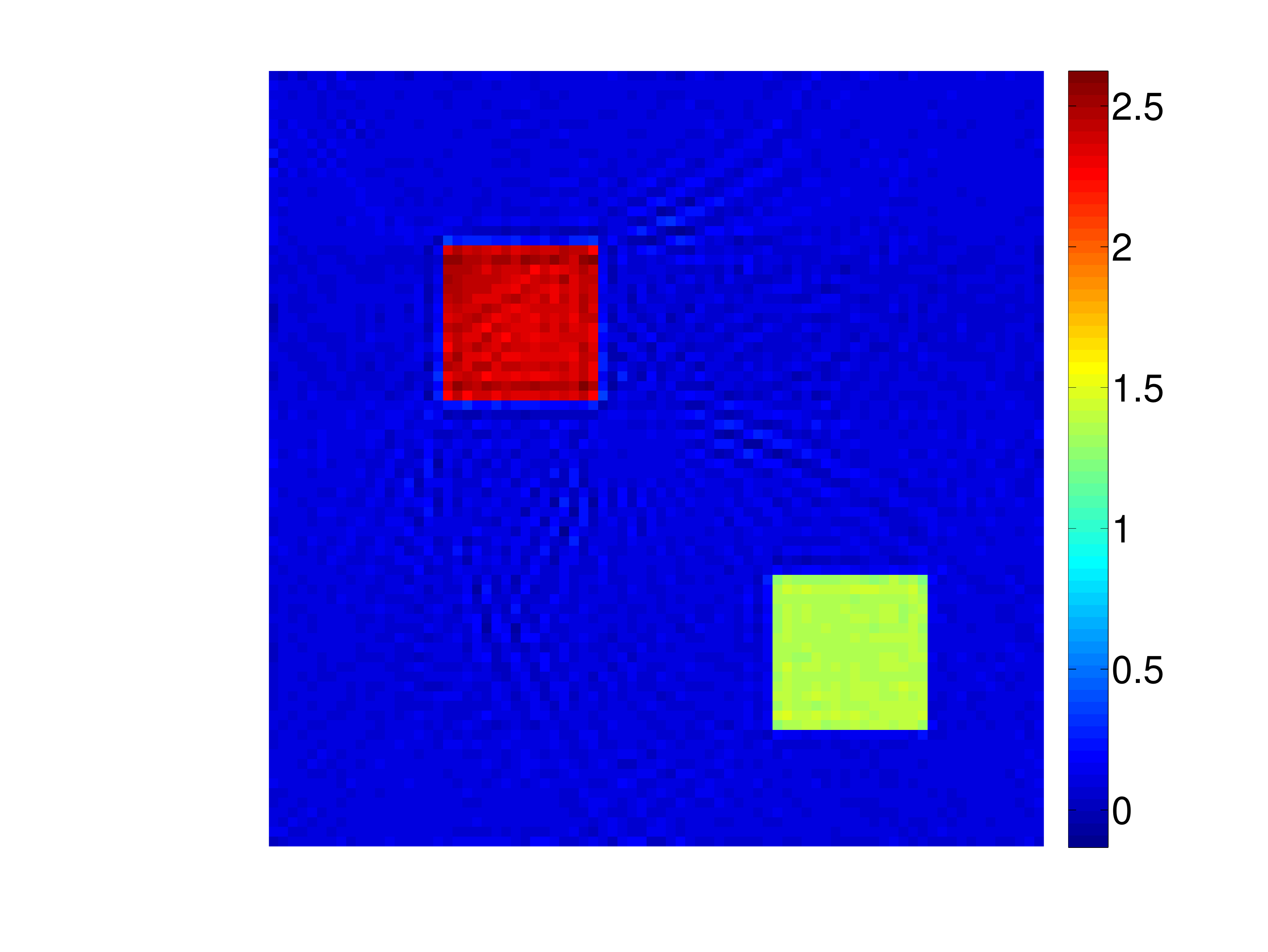}
\caption{4 $\times $ half data. \label{fig:e16}}
\end{subfigure}
\caption{\textsc{Simulation results for low scattering.} \label{fig:1}
The scattering coefficient $\mus=0.1$ is taken constant; the absorption coefficient is shown in \ref{fig:e11}, and the linearization point is given by
$\mua^\lin =  \mus^\lin = 0.1$.}
\end{figure}

\subsubsection*{Small scattering and good contrast in absorbtion}

We first consider a rather small and constant scattering  $\mus=0.1$ and good contrast in the absorption, where  $\mua$ was chosen $0.1$ for the background and $1.5$ respectively  $3$ in  the small boxes in the interior.
Observe that this corresponds to  low scattering regime,  as can be seen by the very well defined shadows behind the obstacles shown in
Figure \ref{fig:1}.  The Landweber iteration \eqref{eq:landweber} has been applied with the linearization point is $\mua^\lin=0.1$ and $\mus^\lin = 0.1$.
The reconstruction for single  illumination  where pressure measurements are made on the whole circle $\partial B_{1.5}(0)$ surrounding the obstacle is shown in Figure \ref{fig:e14}. All  the singularities in $\mua$ are well resolved. Convergence of the Landweber iteration is fast but we can not expect the residuum functional  \eqref{eq:linMin} to go to zero as the data may be outside the range of the linearized forward operator. The reconstruction is qualitatively and quantitatively in good accordance with the phantom.
Figure \ref{fig:e15} shows the  result  for partial acoustic measurements, where the acoustic measurements  are made on a semi-circle on the same side as the illumination. Finally, Figure~\ref{fig:e16} uses four consecutive illuminations with partial  data (again on a semi-circle on the same side as the illumination). This has been implemented by turning the obstacle (or the measurement apparatus)
by $\pi/2$ between consecutive illuminations.
One notices  that for a single illumination the phantom is still quite well resolved, but  the  typical partial data  artifacts can  be observed. These artifacts disappear when incomplete data from multiple measurements are
collected in such a way that the union of the observation sets form the whole circle. 	 Due to the illumination from four sides the reconstruction is even  much better than in the case of one measurement with full data.

\begin{figure}[thb!]\centering
\begin{subfigure}[b]{0.3\linewidth}
\includegraphics[width=\textwidth, trim=4cm 2.5cm 0cm 1cm,clip]{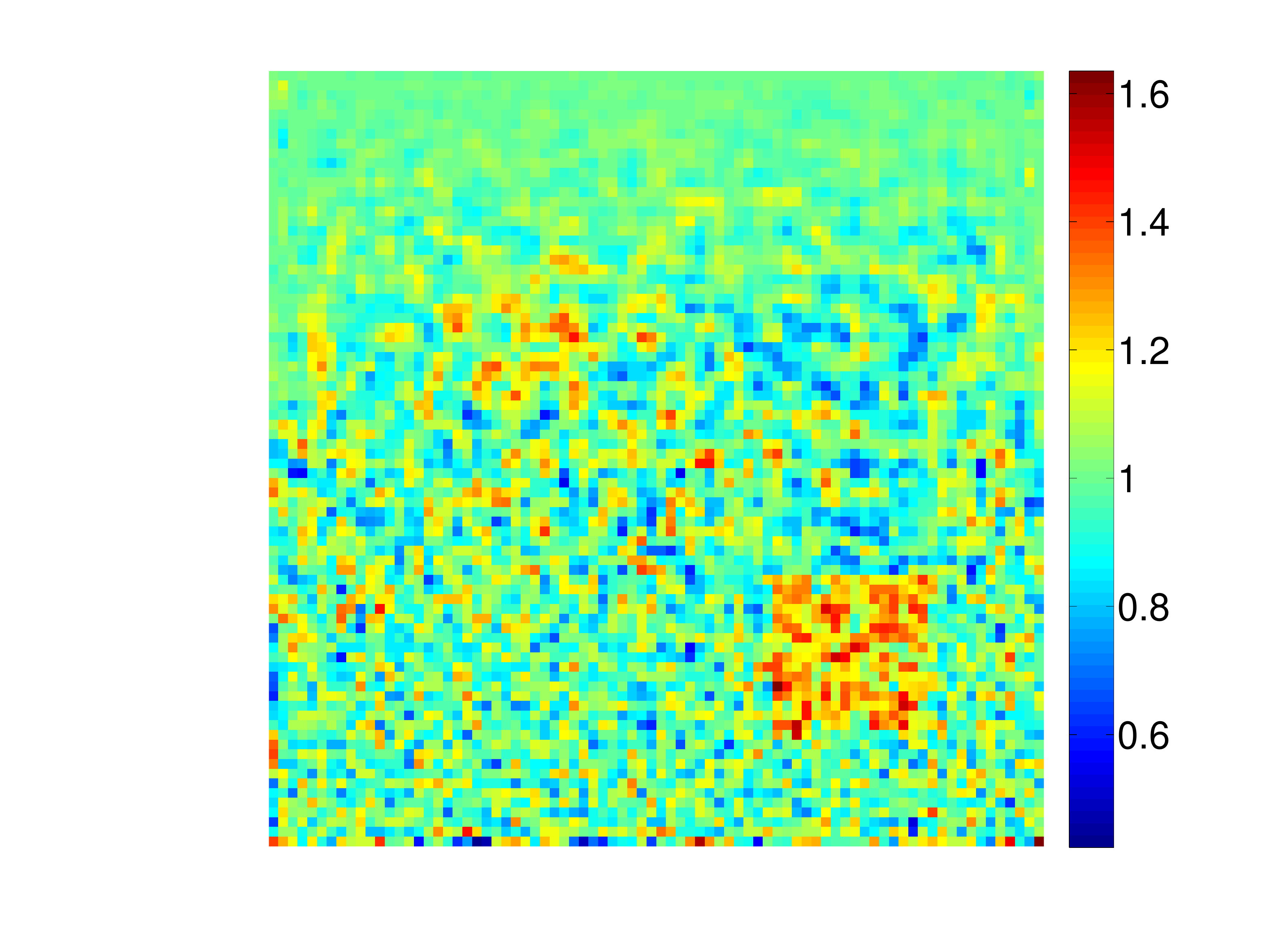}
\caption{Full data}\label{Fig:secondExp2}
\end{subfigure}
\begin{subfigure}[b]{0.3\linewidth}
\includegraphics*[width=\textwidth, trim=4cm 2.5cm 0cm 1cm,clip]{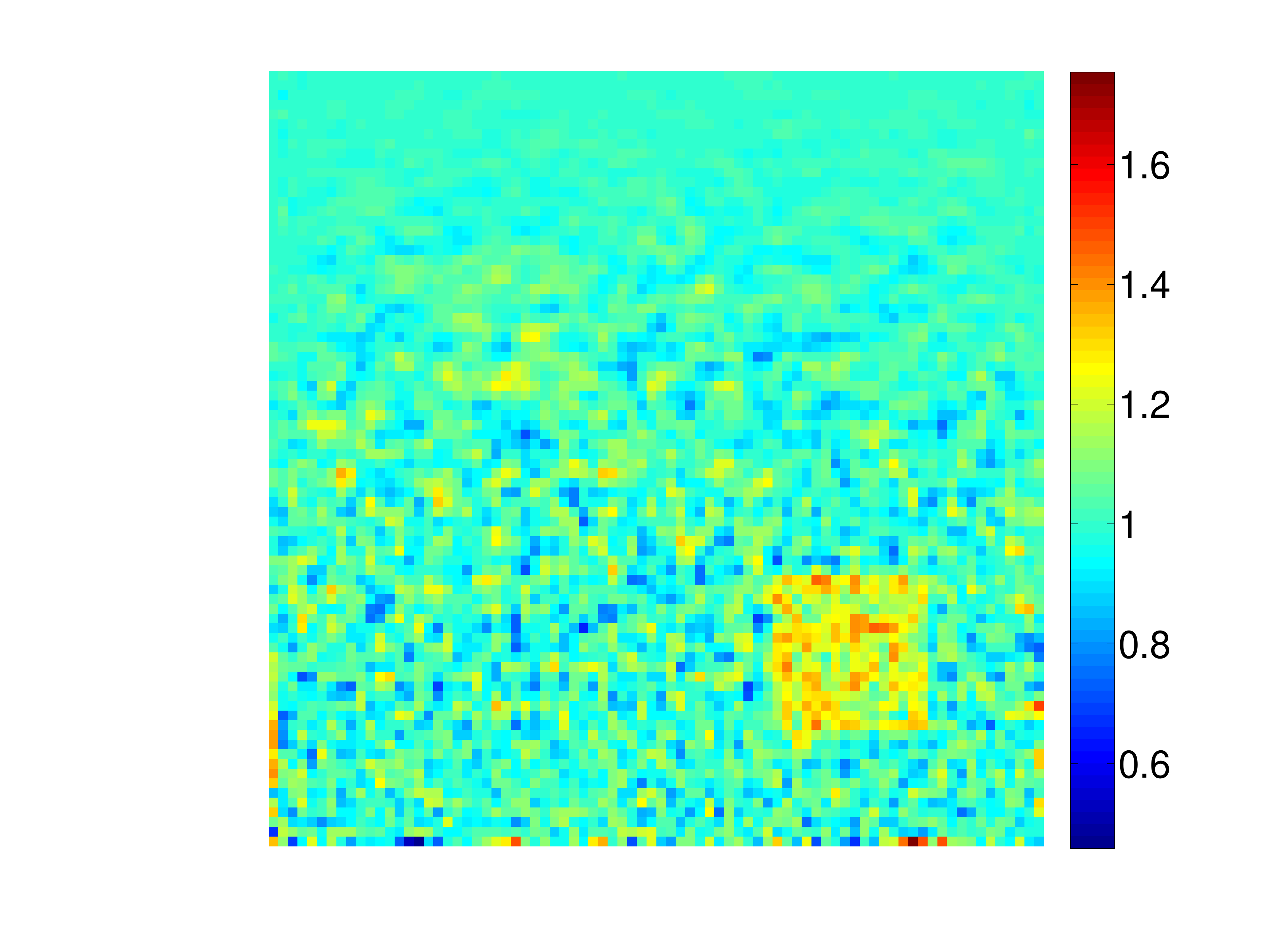}
\caption{Half data}\label{Fig:secondExp3}
\end{subfigure}
\begin{subfigure}[b]{0.3\linewidth}
\includegraphics*[width=\textwidth, trim=4cm 2.5cm 0cm 1cm,clip]{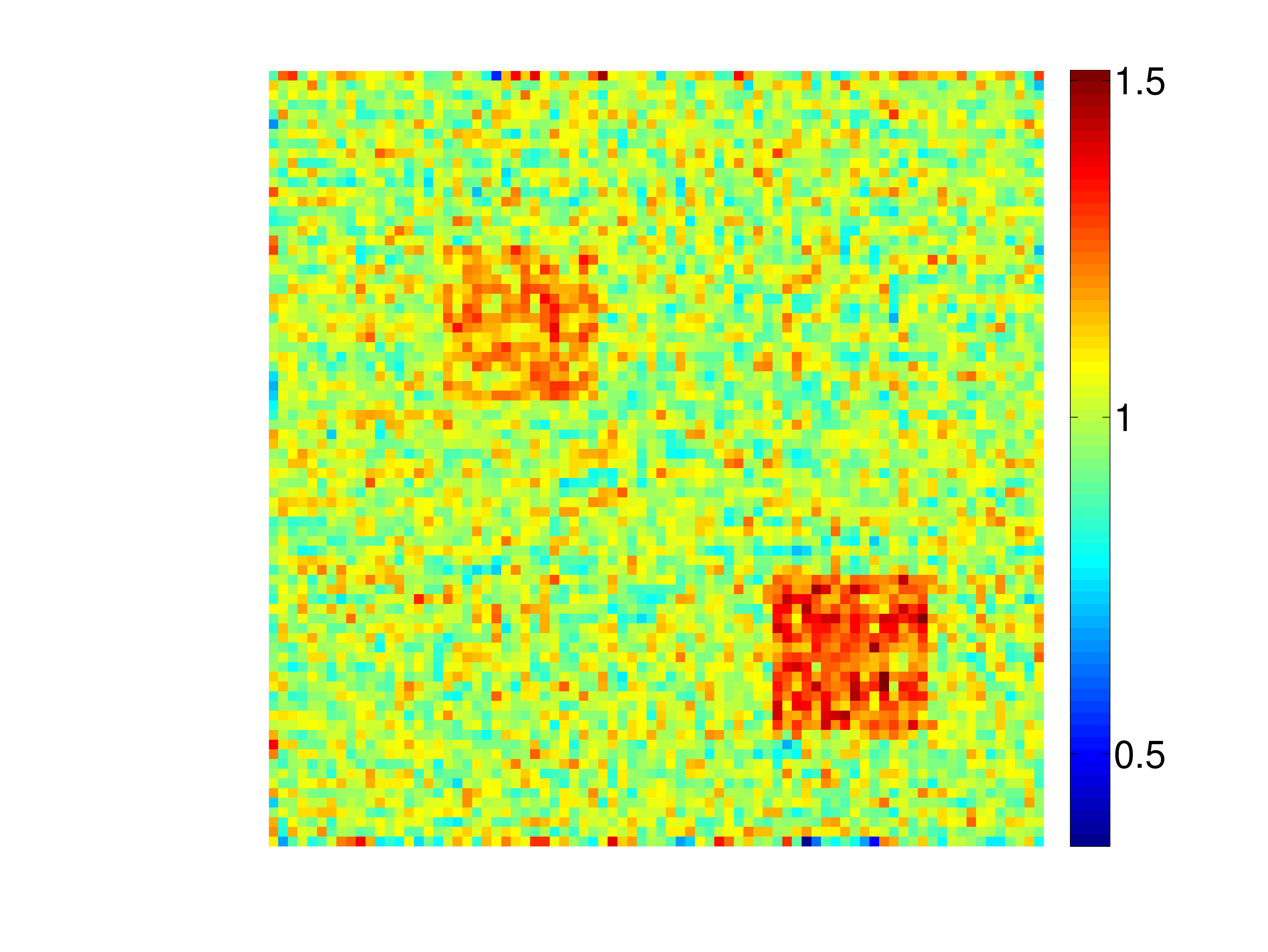}
\caption{4 $\times $ half data}\label{Fig:secondExp4}
\end{subfigure}
\caption{\textsc{Reconstruction\label{Fig:secondFull} of low contrast phantom from data with large noise.}
The actual and linearized scattering parameters have been taken  constant  and equal to
$\mus=\mus^\lin=0.1$.}
\end{figure}

\subsubsection*{Low contrast phantom and large noise}

For  the experiment presented next we decrease the contrast in $\mua$.
To demonstrate the stability of our reconstruction approach  we also increase the noise.
The scattering parameter $\mus=0.1$,
constant throughout the domain,  and the absorption coefficient $\mua$ is chosen to take the value  1 in the background,  and $1.1$ respectively  $1.2$  in the obstacles. The noise has a standard deviation of
$5\%$. The reconstruction results are shown in Figure \ref{Fig:secondFull}.
One notices that the  upper left square of very low contrast is only barely visible in the full
data situation and probably not recognizable if the phantom is unknown. As expected multiple measurements from different directions increase the signal to noise ratio even if only
incomplete data is acquired.

\begin{figure}[thb!]\centering
\begin{subfigure}[b]{0.3\linewidth}
\includegraphics[width=\textwidth, trim=4cm 2.5cm 0cm 1cm,clip]{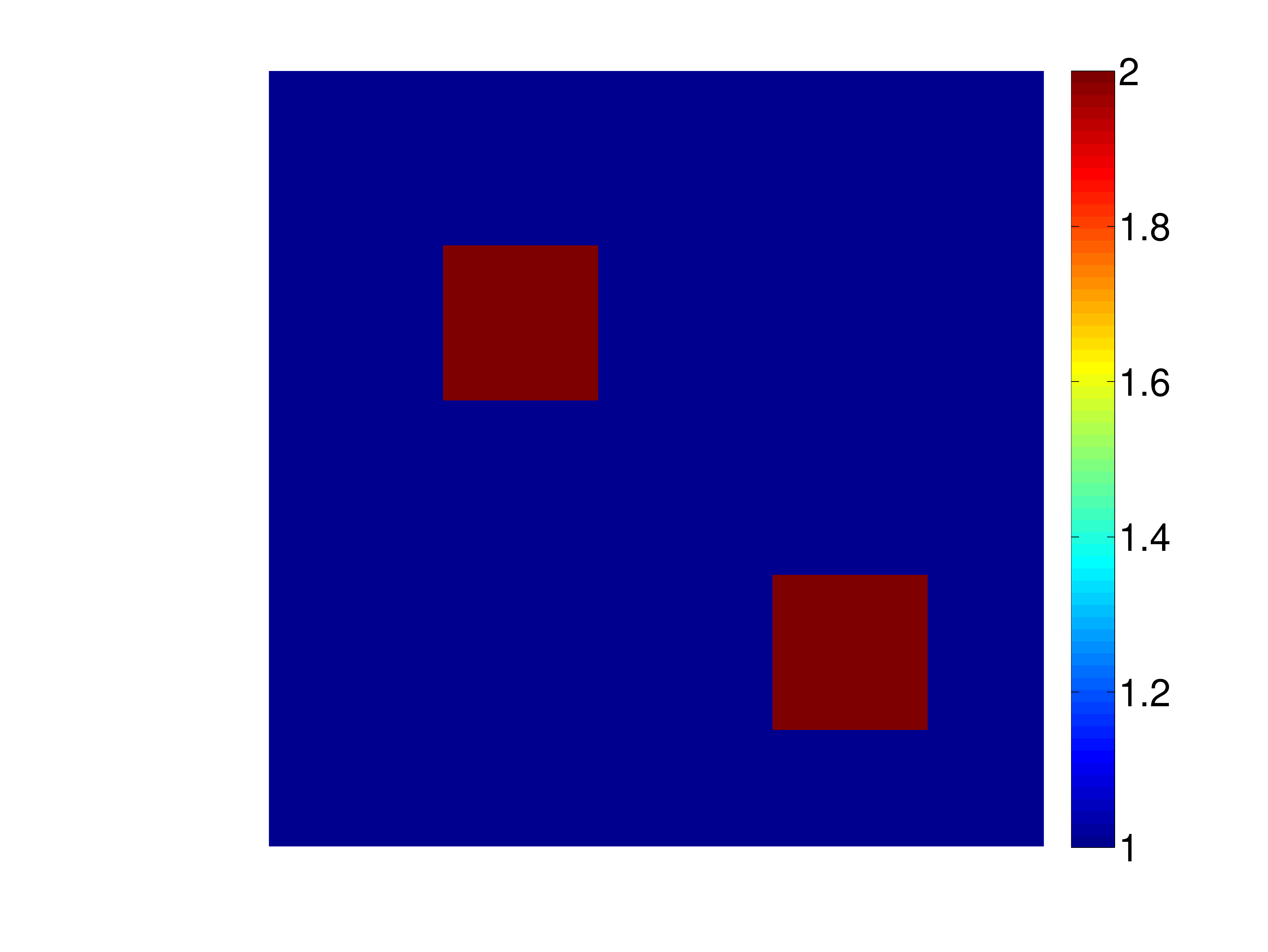}
\caption{Phantom}\label{Fig:ThirdExp1}
\end{subfigure}
\begin{subfigure}[b]{0.3\linewidth}
\includegraphics*[width=\textwidth, trim=4cm 2.5cm 0cm 1cm,clip]{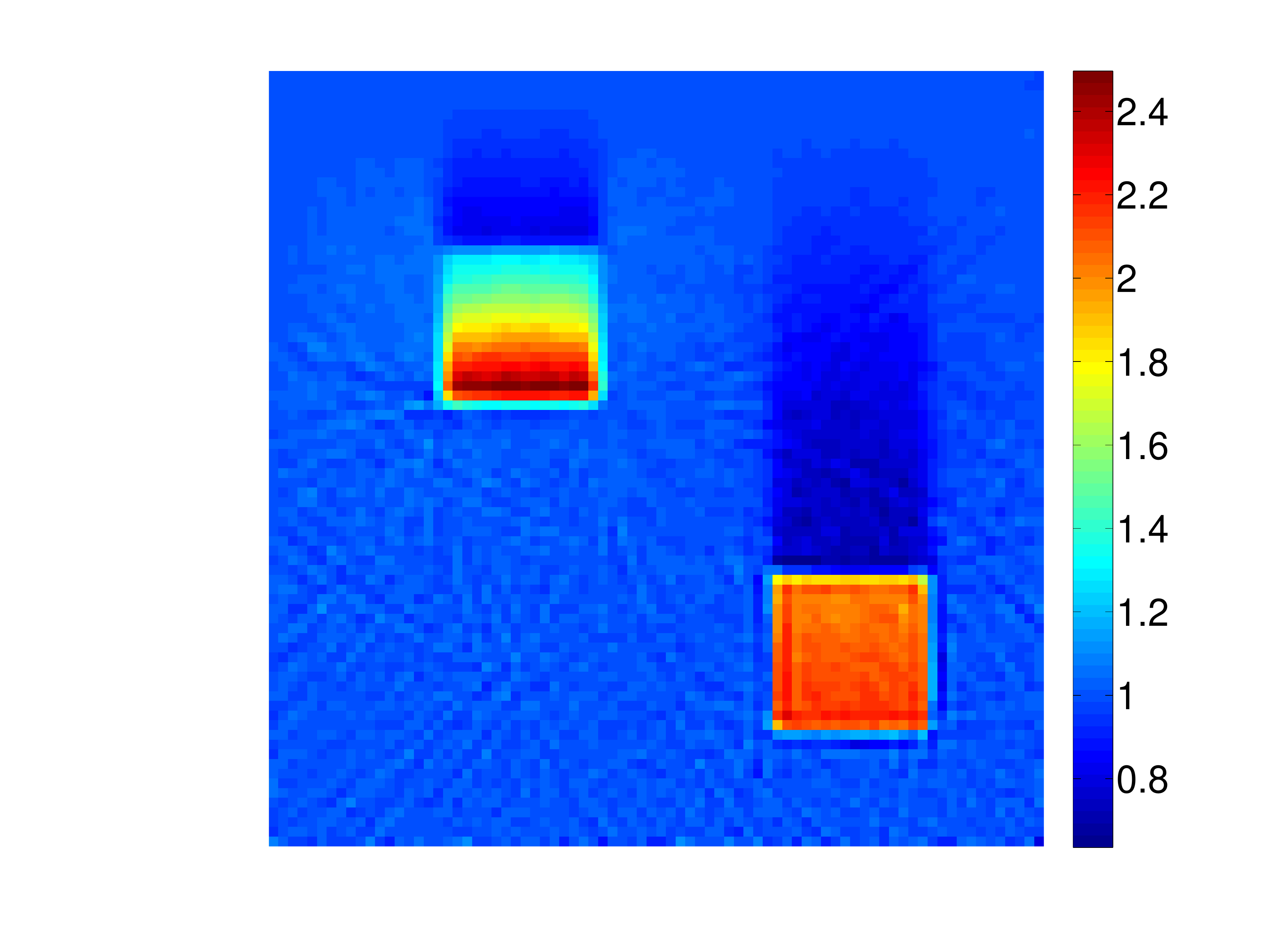}
\caption{Full data}\label{Fig:ThirdExp5}
\end{subfigure}
\begin{subfigure}[b]{0.3\linewidth}
\includegraphics*[width=\textwidth, trim=4cm 2.5cm 0cm 1cm,clip]{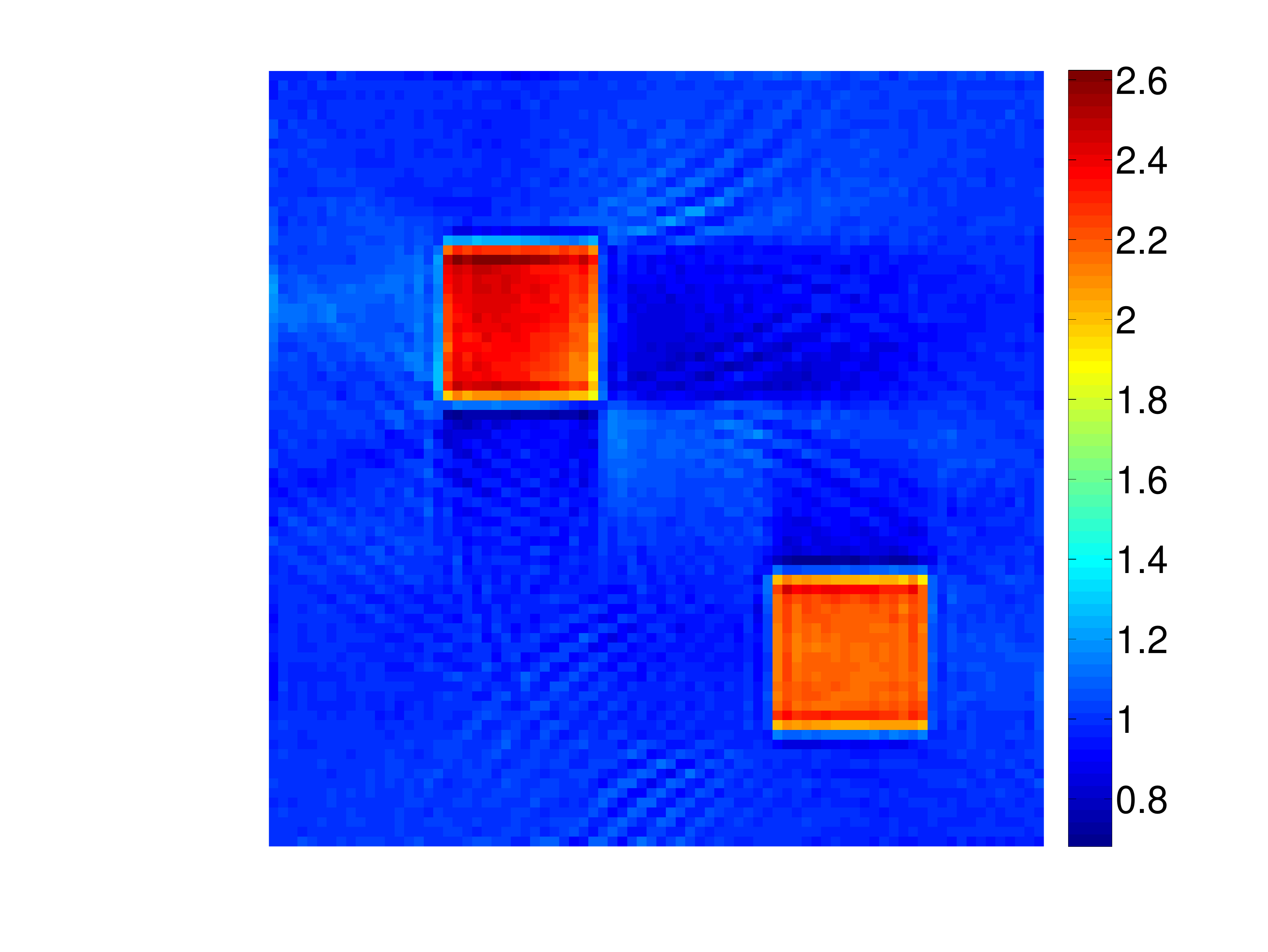}
\caption{4  $\times$  half data}\label{Fig:ThirdExp6}
\end{subfigure}
\caption{\textsc{Reconstruction\label{Fig:thirdFull} for increased scattering.}The actual absorption and scattering coefficients  are shown in the first to pictures and the linearization points have been taken as
$\mua^\lin = 1$  and $\mus^\lin=1$.}
\end{figure}

\subsubsection*{Increased scattering}

In our  final experiment we investigate the  effect of large scattering. We use the value $1$ for $\mus$ in the background,
and the values $1$ and $8$ in the upper left and lower right obstacle, respectively.
The absorption $\mua$ is chosen equal to $1$ in the background and $2$ in the obstacles. To show the consequences of a wrong guess for the linearization point in $\mus$ we choose the linearization point to be $\mua^\lin = 1$  and $\mus^\lin=1$. In particular, the
actual and linearized scattering  coefficients  take different valued in the upper left box.  The noise standard deviation is  taken as   $0.5\%$.

Reconstruction results are shown in Figure \ref{Fig:thirdFull}. It can be seen that also for larger scattering the edges of the squares are well resolved and the quantitative agreement of the reconstruction is still quite good. If the scattering rate in the interior is larger and thus the linearization point is chosen wrong then the algorithm has a tendency to overestimate the absorption. This is due to the fact that a larger scattering rate also leads to larger absorption by lowering the mean free path and thus potentially increasing the length that light has to travel to pass a high scattering area. Note also that in the vicinity of high scattering areas the intensity increases because light is scattered out of that area with larger probability than in the other direction. Thus the absorption in areas close to high scattering areas is underestimated and the absorption inside is overestimated. This phenomenon can be seen quite well in the upper left rectangle in Figure~\ref{Fig:thirdFull}.

\section{Conclusion and outlook}
\label{sec:conclusion}

In this paper we have studied the  linearized inverse  problem of  qPAT
using single as well a multiple illumination. We employed  the RTE
as accurate model for light propagation  in the framework of the
single stage approach introduced in~\cite{haltmeier2015single}.
We have shown that the linearized heating operator $\Do$
is an elliptic  pseudodifferential  operator of order zero provided
that  the background fluence is non-vanishing
(see Theorems~\ref{thm:no-scattering} and \ref{thm:scatt}). This in particular implies
the stability of the generalized (Moore-Penrose) inverse  of $\Do$. Further, we were able to
show injectivity and  two-sided stability estimates for the linearized inverse problem  using single as well a multiple illuminations. These results are presented in  Theorems~\ref{thm:sigma0-stablity} and~\ref{thm:no} for non-vanishing scattering, in Theorem~\ref{thm:scatt}) for vanishing scattering, and  in Theorem~ \ref{thm:multiple} for multiple illuminations.
In the case of non-vanishing scattering  our condition  guaranteeing injectivity requires quite small values of  scattering and absorption at the linearization point.
Relaxing such assumptions is an interesting  line of future research. Another important
aspect is the extension of our  stability estimates to the fully
non-linear case.  Finally, investigating single state qPAT with multiple illuminations for moving object (see \cite{chungmotion} for qualitative PAT with moving object) is also a challenging topic.

For numerical computations, the linearization simplifies matters considerable because
all the matrices for the solution of the RTE have to be constructed only once.
Detailed  numerical simulations have  been performed to demonstrate the
feasibility of solving the linearized problem.  From the presented  numerical results
we conclude that solving the linearized inverse problem gives useful quantitative reconstructions
even if no attempt to reconstruct the scattering coefficient has been made.  Recovering the
absorption  and the scattering coefficient simultaneously seems difficult because the scattering
leads to a lower order contribution to the data, as can also be seen from the theoretical considerations in Section~\ref{sec:stability}. Nonetheless suitable regularization  and preconditioning can lead to
reasonable reconstruction results.  Such investigations will also be subject of further work.

\section*{Acknowledgement}

L. Nguyen's research is partially supported by the NSF grants DMS 1212125 and DMS 1616904. He also thanks the University of Innsbruck  for financial support and hospitality during his visit. The work of S.~Rabanser has partially been supported by a doctoral fellowship from the University of Innsbruck.

\appendix
\section{Inversion of the  wave equation}
\label{sec:aux}

Recall the operator $\Wo_{\La, T}$ is defined by
$\Wo_{\La, T} h = p|_{\Lambda \times (0,T)}$,
where $p$ is the solution of\eqref{eq:wave-fwd} with initial data
$h$ supported inside $\Om$ and $\Lambda \subseteq \mS$ is a subset
of the closed surface  enclosing $\Om$. Lemma~\ref{lem:contW} states
that  $\Wo_{\La, T} \colon L^2_\Om (\R^d) \to  L^2\kl{\Lambda \times (0, T)}$
is linear and bounded.
The inversion of $\Wo_{\mS, T}$ is well  studied and several efficient inversion algorithm are available. Such algorithms include explicit  inversion formulas~\cite{FinHalRak07,FinPatRak04,Kun07a,Hal14,HalPer15b,Nat12,Ngu09,XuWan05},
series solution \cite{AgrKuc07,Hal09,Kun07b,XuXuWan02},
time reversal methods~\cite{BurMatHalPal07,HriKucNgu08,FinPatRak04,SteUhl09}, and
 iterative approaches based on the adjoint \cite{arridge2016adjoint,belhachmi2016direct,haltmeier2016iterative,huang2013full,RosNtzRaz13}.  For the limited data case, most of the reconstruction methods are of iterative nature (see, e.g., \cite{arridge2016adjoint,belhachmi2016direct,haltmeier2016iterative,huang2013full,RosNtzRaz13}).  In this appendix, we describe some theoretical results on the inversion of $\Wo_{\La,T}$ that are relevant for our purpose.

\begin{proposition}[Uniqueness of reconstruction] \label{prop:wave-uniqueness}
The data $\Wo_{\La, T}(h)  $ uniquely determines $h$ on $\Om_{\La,T}
\coloneqq \set{x \in \Om \mid  \dist(x,\Lambda) \leq T}$.
\end{proposition}

\begin{proof}
See \cite{FinPatRak04,SteUhl09}.
\end{proof}

Proposition~\ref{prop:wave-uniqueness}  in particular implies that $\Wo_{\La, \Omega}$ is injective from $L^2(\Om)$ to $L^2(\Lambda \times (0,T))$ if $\Om_{\La,T} = \Om$. This holds, for example, when $T \geq \max_{x\in \Om} \dist(x,\La)$. However, even if the inverse operator  $\Wo^{-1}$ exists, its computation  may be a severely ill-posed problem.  Stability of the reconstruction can be obtained if additionally the following visibility condition is satisfied.

\begin{condition}[Visibility condition] \label{A:Visible}
For each $x \in \Om$ a line passing through $x$ intersects
$\Lambda$ at a point of distance less than $T$ from $x$.
\end{condition}

Under the visibility condition the following stability result holds.

\begin{proposition}[Stability of inversion] \label{prop:wave-stability}
If the visibility condition \ref{A:Visible} holds, then
$\Wo^{-1}_{\La,T} \colon L^2(\Lambda \times (0,T)) \to L^2(\Om)$ is bounded.
\end{proposition}

\begin{proof}
See \cite{SteUhl09}.
\end{proof}

Proposition~\ref{prop:wave-stability} in particular implies that under the
visibility condition, the inverse operator $\Wo^{-1}_{\La,T}$ is Lipschitz continuous.
On the other hand, if the visibility condition does not hold, then
$\Wo^{-1}_{\La,T}$ is not even conditionally H\"older continuous (see \cite{nguyen2011singularities}).

Although it is not clear how to directly evaluate $\Wo_{\La,T}^{-1}$ for partial data, microlocal inversion of $\Wo_{\La,T}$ is quite straight forward. That is, one can recover the visible singularities of $h$ from $\Wo_{\La,T}(h)$ by a direct method.  To this end,  let $\chi_{\La,T} \in C^\infty(\mS \times [0,\infty))$ be a nonnegative function with  $\supp(\chi_{\La,T}) = \Lambda \times [0,T]$. Then, one can decompose
$$\chi_{\La,T} \Wo_{\La,T}  h   = \Wo^{(+)} h + \Wo^{(-)} h\,,$$
where
$\Wo^{(+)}, \Wo^{(-)} \colon C^\infty(\Om) \to C^\infty(\Lambda \times (0,T))$
are Fourier integral operators of order zero. Each of them describes how the singularities of $h$ induce singularities of $\Wo_{\La, T} h$. Each singularity $(x,\xi) \in \wf(h)$ breaks into two equal parts traveling on opposite rays $\ell_\pm(x,\xi)  \coloneqq  \{x \pm r \xi \mid  r>0\}$. Those singularities hit the observation surface $\mS$ at location $y_\pm(x,\xi)$ and time $t_\pm(x,\xi) = \abs{x- y_\pm(x,\xi)}$. Their projection on the cotangent bundle of $\mS \times (0,\infty)$ at $(y_\pm(x,\xi), t_\pm(x,\xi))$ are the induced singularities of $\chi_{\La,T} \Wo_{\La, T}h$ if $\chi_{\La,T}(y_+(x,\xi),t_+(x,\xi))>0$ or $\chi_{\La,T}(y_-(x,\xi),t_-(x,\xi))>0$. In that case, $(x,\xi)$ is called a visible singularity of $h$.
Any visible singularity can be reconstructed by time-reversal, described as follows.
For given data $g \colon \mS \times (0, T) \to \R$,
consider the  time-reversed problem
\begin{equation} \label{eq:wave-back}
	\left\{ \begin{aligned}
	 \partial_t^2 q (x,t) - \Delta q(x,t)
	&=
	0 \,,
	 && \text{ for }
	\kl{x,t} \in
	D \times \kl{0, T}
	\\
	q\kl{x,T}=	\partial_t q\kl{x,T}
	&=
	0 \,,
	&& \text{ for }
	x  \in D
	\\
	q\kl{x,t}
	&=g(x,t) \,,
	&& \text{ for }
	(x,t)  \in \mS \times (0,T) \,.
\end{aligned} \right.
\end{equation}
We define the time reversal operator $\Ao$ by $\Ao  ( g ) \coloneqq q(\edot,0)$.

\begin{lemma}[Recovery of singularities] \label{T:SU}  $\Ao \chi_{\La,T} \Wo_{\La,T} $ is a pseudo-differential operator of order zero, whose principal symbol is $\tfrac{1}{2}  \sum_{\sigma = \pm} \chi_{\La,T}(y_\sigma(x,\xi),t_\sigma(x,\xi))$.
\end{lemma}

\begin{proof}
See \cite{SteUhl09}.
\end{proof}

Let $(x,\xi)$ be a visible singularity of $h$. Since $\Ao \chi_{\La,T} \Wo_{\La,T}$ is positive at $(x,\xi)$, $(x,\xi)$ is also a singularity of $\Ao \chi_{\La,T} \Wo_{\La,T}(h)$. That is, all the visible singularities are reconstructed by the time-reversal method.
We, finally, note that the multiplication with a smooth function $\chi_{\La,T}$ is essential, since otherwise the time-reversal procedure introduces artifacts,
see~\cite{frikel2015artifacts,nguyen2011singularities}.

\end{document}